\documentclass[a4paper, 11pt]{article}
\usepackage[utf8]{inputenc}
\usepackage[T1]{fontenc}
\usepackage[english]{babel}
\usepackage{amssymb, amsthm, mathtools, mathrsfs}
\usepackage{geometry}
\usepackage{tikz}
	\tikzstyle{lien}=[->,>=stealth,rounded corners=2pt,thick]
	\tikzset{individu/.style={draw,#1},individu/.default={}}
\usepackage{enumitem}

\usepackage{hyperref}
\hypersetup{
	hidelinks,
	breaklinks=true,
	bookmarksopen=true,
	bookmarksnumbered,
	pdftitle = {On the sizes of burnt and fireproof components for fires on a large Cayley tree},
	pdfauthor = {Cyril Marzouk},
	pdfsubject = {Probability},
	pdfkeywords={probability fire Cayley trees}
}
\addto\extrasenglish{%
}
\newcommand{\ensembles}[1]{\mathbf{#1}}
	\newcommand{\N}{\ensembles{N}}
	
	\newcommand{\R}{\ensembles{R}}
\renewcommand{\P}{\mathrm{P}}
\newcommand{\E}{\mathrm{E}}
\newcommand{\ind}[1]{\mathbf{1}_{\{#1\}}}
\newcommand{\ex}{\mathrm{e}}
\newcommand{\e}{\mathbf{e}}
\newcommand{\eps}{\varepsilon}
\renewcommand{\d}{\mathrm{d}}
\newcommand{\Card}{\mathrm{Card}}
\renewcommand{\b}{\mathrm{b}}
\renewcommand{\t}{\mathbf{t}}
\newcommand{\tn}{\t_n}
\newcommand{\Cut}{\mathrm{Cut}}
\newcommand{\Cutn}{\Cut(\tn)}
\newcommand{\CRT}{\mathcal{T}}
\newcommand{\GWT}{\mathbb{T}}
\newcommand{\GWTn}{\GWT_n}
\newcommand{\Red}{\mathcal{R}}
\newcommand{\Gauss}{Z}
\theoremstyle{plain}
	\newtheorem{Thm}{\textbf{Theorem}}
	\newtheorem{Pro}{\textbf{Proposition}}
	\newtheorem{Lem}{\textbf{Lemma}}
\theoremstyle{definition}
	\newtheorem{Rem}{\textbf{Remark}}
\makeatletter \renewenvironment{proof}[1][\proofname] {\par\pushQED{\qed}\normalfont\topsep6\p@\@plus6\p@\relax\trivlist\item[\hskip\labelsep\bfseries#1\@addpunct{.}]\ignorespaces}{\popQED\endtrivlist\@endpefalse} \makeatother
\title{On the sizes of burnt and fireproof components for fires on a large Cayley tree}
\author{Cyril Marzouk
	\thanks{Institut f\"{u}r Mathematik, Universit\"{a}t Z\"{u}rich, Winterthurerstrasse 190, CH-8057 Zürich, Switzerland. Email: \href{mailto:cyril.marzouk@math.uzh.ch}{\texttt{cyril.marzouk@math.uzh.ch}}.}}
\date{}
\begin{document}
\maketitle

\begin{abstract}
We continue the study initiated by Jean Bertoin in 2012 of a random dynamics on the edges of a uniform Cayley tree with $n$ vertices in which, successively, each edge is either set on fire with some fixed probability $p_n$ or fireproof with probability $1-p_n$. An edge which is set on fire burns and sets on fire its flammable neighbors, the fire then propagates in the tree, only stopped by fireproof edges. We study the distribution of the proportion of burnt and fireproof vertices and the sizes of the burnt or fireproof connected components as $n \to \infty$ regarding the asymptotic behavior of $p_n$.
\end{abstract}


\section{Introduction and main results}\label{introduction}

We recall the definition of the fire dynamics introduced by Bertoin \cite{Bertoin-Fires_on_trees}. Given a tree of size $n$ and a number $p_n \in [0,1]$, we consider the following random dynamics: initially every edge is flammable, then successively, in a random uniform order, each edge is either fireproof with probability $1-p_n$ or set on fire with probability $p_n$. In the latter case, the edge burns, sets on fire its flammable neighbors and the fire propagates instantly in the tree, only stopped by fireproof edges. An edge which has been burnt because of the propagation of fire is not subject to the dynamics thereafter. The dynamics continue until all edges are either burnt or fireproof. A vertex is called fireproof if all its adjacent edges are fireproof and called burnt otherwise; we discard fireproof edges that have at least one burnt extremity and thus get two forests: one consists of fireproof trees and the other of burnt trees. See \autoref{fig-arbre_brule} for an illustration. We study the asymptotic behavior of the size of these two forests and of their connected components as the total size of the tree tends to infinity.

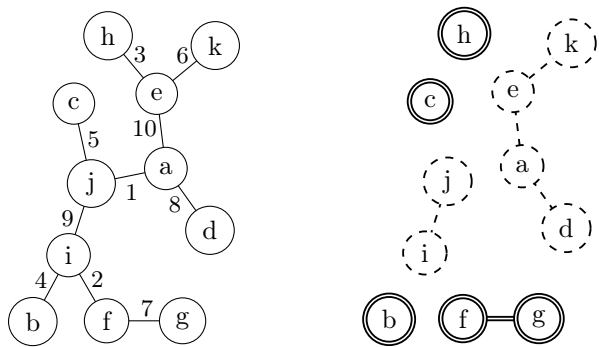
\begin{figure}[ht] \centering
\begin{tabular}{ccc}
\begin{tikzpicture}
\draw
	(-0.2,1.82)-- (0.78,2.02)		
	(0,0)-- (-0.5,0.87)			
	(0.66,3.01)-- (0.02,3.78)		
	(-0.5,0.87)-- (-0.97,-0.01)		
	(-0.2,1.82)-- (-0.43,2.88)		
	(1.43,3.65)-- (0.66,3.01)		
	(1,0)-- (0,0)				
	(0.78,2.02)-- (1.36,1.2)		
	(-0.5,0.87)-- (-0.2,1.82)		
	(0.66,3.01)-- (0.78,2.02)		
;
\begin{small}
\draw
	(0.78,2.02)	node (a) [circle, fill=white, draw]	{a}
	(-0.97,-0.01)	node (b) [circle, fill=white, draw]	{b}
	(-0.43,2.88)	node (c) [circle, fill=white, draw]	{c}	
	(1.36,1.2)		node (d) [circle, fill=white, draw]	{d}
	(0.66,3.01)	node (e) [circle, fill=white, draw]	{e}
	(0,0)			node (f) [circle, fill=white, draw]	{f}
	(1,0)			node (g) [circle, fill=white, draw]	{g}
	(0.02,3.78)	node (h) [circle, fill=white, draw]	{h}
	(-0.5,0.87)		node (i) [circle, fill=white, draw]	{i}
	(-0.2,1.82)		node (j) [circle, fill=white, draw]	{j}
	(1.43,3.65)	node (k) [circle, fill=white, draw]	{k}
;
\end{small}
\begin{footnotesize}
\draw
	(0.33,1.70)	node		{1}
	(-0.12,0.55)	node		{2}
	(0.44,3.53)	node		{3}
	(-0.86,0.54)	node		{4}
	(-0.17,2.4)		node		{5}
	(1,3.52)		node		{6}
	(0.53,0.17)	node		{7}
	(0.9,1.55)		node		{8}
	(-0.51,1.36)	node		{9}
	(0.5,2.55)		node		{10}
;
\end{footnotesize}
\end{tikzpicture}
&
\hspace*{2em}
&
%
\begin{tikzpicture}
\draw [line width=.7pt, double]
	(1,0)-- (0,0)				
;
\draw [line width=.7pt, dashed]
	(1.43,3.65)-- (0.66,3.01)		
	(0.78,2.02)-- (1.36,1.2)		
	(-0.5,0.87)-- (-0.2,1.82)		
	(0.66,3.01)-- (0.78,2.02)		
;
\begin{small}
\draw
	(0.78,2.02)	node (a) [circle, line width=.7pt, fill=white, draw, dashed]	{a}
	(-0.97,-0.01)	node (b) [circle, line width=.7pt, double, fill=white, draw]	{b}
	(-0.43,2.88)	node (c) [circle, line width=.7pt, double, fill=white, draw]	{c}	
	(1.36,1.2)		node (d) [circle, line width=.7pt, fill=white, draw, dashed]	{d}
	(0.66,3.01)	node (e) [circle, line width=.7pt, fill=white, draw, dashed]	{e}
	(0,0)			node (f) [circle, line width=.7pt, double, fill=white, draw]	{f}
	(1,0)			node (g) [circle, line width=.7pt, double, fill=white, draw]	{g}
	(0.02,3.78)	node (h) [circle, line width=.7pt, double, fill=white, draw]	{h}
	(-0.5,0.87)		node (i) [circle, line width=.7pt, fill=white, draw, dashed]	{i}
	(-0.2,1.82)		node (j) [circle, line width=.7pt, fill=white, draw, dashed]	{j}
	(1.43,3.65)	node (k) [circle, line width=.7pt, fill=white, draw, dashed]	{k}
;
\end{small}
\end{tikzpicture}
\end{tabular}
\caption{Given a tree and a uniform enumeration of its edges on the left, if the edges set on fire are the 6th and the 9th, we get the two forests on the right where dotted lines stand for "burnt" and double lines for "fireproof".}
\label{fig-arbre_brule}
\end{figure}

In this work, we assume that the tree is a uniform Cayley tree of size $n$, denoted by $\tn$, i.e. a tree picked uniformly at random amongst the $n^{n-2}$ different trees on a set of $n$ labeled vertices, say, $[n] = \{1, \dots, n\}$. For this model, the system exhibits a phase transition as it is shown by Bertoin \cite{Bertoin-Fires_on_trees}. Theorem 1 in \cite{Bertoin-Fires_on_trees} is stated in the case where $p_n \sim c n^{-\alpha}$ with $c, \alpha > 0$ but extends verbatim as follows: denote by $I_n$ and $B_n$ respectively the total number of fireproof and burnt vertices of $\tn$, then
\begin{enumerate}
\item\label{thm-transition_de_phase_sous} If $\lim_{n \to \infty} n^{1/2} p_n = \infty$ (subcritical regime), then $\lim_{n \to \infty} n^{-1} I_n=0$ in probability.
\item\label{thm-transition_de_phase_sur} If $\lim_{n \to \infty} n^{1/2} p_n =0$ (supercritical regime), then $\lim_{n \to \infty} n^{-1} B_n=0$ in probability.
\item\label{thm-transition_de_phase_crit} If $\lim_{n \to \infty} n^{1/2} p_n = c$ for some $c > 0$ (critical regime), then $\lim_{n \to \infty} n^{-1} I_n = D(c)$ in distribution where
\begin{equation}\label{eq-distribution_D_infini}
\P(D(c) \in \d x) = \frac{c}{\sqrt{2 \pi x (1-x)^3}} \exp\bigg(-\frac{c^2 x}{2 (1-x)}\bigg) \d x,
\qquad 0 < x < 1.
\end{equation}
\end{enumerate}
The aim of this paper is to improve these three convergences. For the first two regimes, we prove a convergence in distribution under an appropriate scaling of $I_n$ and $B_n$ respectively, see the statements below. For the critical regime, we prove the joint convergence in distribution of the number of fireproof vertices and the sizes of the burnt subtrees, ranked in non-increasing order; the precise statement requires some notations and is postponed to \autoref{section3}, see \autoref{thm-crit} there. We next state our main result concerning the subcritical regime.

\begin{Thm}\label{thm-sous_crit}
Suppose that $\lim_{n \to \infty} n^{1/2} p_n = \infty$. Then
\begin{equation*}
\lim_{n \to \infty} p_n^2 I_n = \Gauss^2
\qquad\text{in distribution,}
\end{equation*}
where $\Gauss$ is a standard Gaussian random variable.
\end{Thm}

Consider then the supercritical regime; as we are interested in the asymptotic behavior of $B_n$ we assume that $p_n \gg n^{-1}$ so that the probability that all the vertices are fireproof is $(1 - p_n)^{n-1} \to 0$.
\begin{Thm}\label{thm-sur_crit_1}
Suppose that $\lim_{n \to \infty} n^{1/2} p_n = 0$ and $\lim_{n \to \infty} n p_n = \infty$. Then
\begin{equation*}
\lim_{n \to \infty} (n p_n)^{-2} B_n = \Gauss^{-2}
\qquad\text{in distribution,}
\end{equation*}
where $\Gauss$ is a standard Gaussian random variable.
\end{Thm}

\begin{Rem}
Let $\Gauss$ be a standard Gaussian random variable. One can check from \eqref{eq-distribution_D_infini} that
\begin{equation}\label{eq-limite_D_c_en_fonction_du_parametre}
\lim_{c \to \infty} c^2 D(c) = \Gauss^2
\quad\text{and}\quad
\lim_{c \to 0} c^{-2} (1 - D(c)) = \Gauss^{-2}
\end{equation}
in distribution. Very informally, if we write $I_n(p_n)$ for the number of fireproof vertices of $\tn$ when the probability to set on fire a given edge is $p_n$, and similarly $B_n(p_n)$, then \ref{thm-transition_de_phase_crit} above shows that for every $c \in (0, \infty)$ fixed,
\begin{equation*}
I_n(c n^{-1/2}) \approx n D(c),
\quad\text{and}\quad
B_n(c n^{-1/2}) \approx n (1 - D(c)).
\end{equation*}
From \eqref{eq-limite_D_c_en_fonction_du_parametre}, one is tempted to write more generally for $p_n \gg n^{-1/2}$,
\begin{equation*}
I_n(p_n) \approx n D(n^{1/2} p_n) \approx p_n^{-2} \Gauss^2,
\end{equation*}
and for $p_n \ll n^{-1/2}$,
\begin{equation*}
B_n(p_n) \approx n (1 - D(n^{1/2} p_n)) \approx (n p_n)^2 \Gauss^{-2}.
\end{equation*}
However, it does not seem clear to the author how to prove respectively \autoref{thm-sous_crit} and \autoref{thm-sur_crit_1} from this sketch. Indeed the argument in \cite{Bertoin-Fires_on_trees} does not enable one to deal with the sub or supercritical regime and the proofs given here are different from that of \ref{thm-transition_de_phase_sous}, \ref{thm-transition_de_phase_sur} and \ref{thm-transition_de_phase_crit}.
\end{Rem}

The rest of this paper is organized as follows. Relying on Pitman \cite{Pitman-Poisson_Kingman_partitions-test} and Chaumont and Uribe Bravo \cite{Chaumont_Uribe_Bravo-Markovian_bridges_weak_continuity_and_pathwise_constructions}, we briefly discuss in \autoref{section2} the existence of a conditional distribution for the sequence of the ranked sizes of the jumps made during the time interval $[0, 1]$ by a certain subordinator, say, $\sigma$, conditionally given the value of the latter at time $1$. We also prove the continuity of this conditional distribution in the terminal value $\sigma(1)$, which will be used to derive our first result.

We then focus on the critical regime in \autoref{section3}. Motivated by the proof of \ref{thm-transition_de_phase_crit} in Bertoin \cite{Bertoin-Fires_on_trees}, we associate with the Cayley tree $\tn$ its cut-tree and view the fire dynamics as a point process on the latter. Using the ideas of Aldous and Pitman \cite{Aldous_Pitman-Standard_additive_coalescent}, we show that the marked cut-tree converges to the Brownian continuum random tree (CRT) endowed with a slight modified version of the point process obtained in \cite{Aldous_Pitman-Standard_additive_coalescent}. This yields the joint convergence of the number of fireproof vertices and the sizes of the burnt connected components to the masses of the components of the CRT logged at the atoms of the point process. Using a second approximation of the CRT with finite trees, we further express this limit as a mixture of the jumps of the previous subordinator $\sigma$ conditioned on the value of the latter at time 1, with a mixing law $D(c)$.

We prove \autoref{thm-sous_crit} in \autoref{section4}. For this, we shall see that, with high probability, the remaining forest after the first fire has a total size of order $p_n^{-2}$ and so have its largest trees. Note that the dynamics then continue on each subtree independently. Informally, the smallest ones do not contribute much and may be neglected, while the dynamics on the largest subtrees are now critical. A slight generalization of \ref{thm-transition_de_phase_crit} then yields an asymptotic for the number of fireproof vertices in each subtree and so for the total number of fireproof vertices.

Finally, we prove \autoref{thm-sur_crit_1} in \autoref{section5}. Consider the sequence of the sizes of the burnt subtrees, ranked in order of appearance, and all rescaled by a factor $(n p_n)^{-2}$. We prove that the latter converges in distribution for the $\ell^1$ topology, from which \autoref{thm-sur_crit_1} follows readily. To this end, we first show for every integer $j$ the joint convergence for the size of the $j$ first burnt subtrees; then we show that, taking $j$ large enough, the next trees are arbitrary small.


\section{Preliminaries on subordinators and bridges}\label{section2}

Let $(\sigma(t), t \ge 0)$ be the first-passage time process of a linear Brownian motion: $\sigma$ is a stable subordinator of index $1/2$ such that
\begin{equation}\label{eq-definition_subordinateur_S}
\E[\exp(- q \sigma(t))] = \exp(-t \sqrt{2 q}),\quad \text{for any } t, q \ge 0.
\end{equation}
Let $J_1 \ge J_2 \ge \dots \ge 0$ be the ranked sizes of its jumps made during the time interval $[0, 1]$. We need to make sense of the conditional distribution of the sequence $(J_i)_{i \ge 1}$ conditionally given the null event $\{\sigma(1) = z\}$ in the set $\ell^1(\R)$ of real-valued summable sequences.

From the Lévy-It\={o} decomposition, we know that $\sigma$ is a right-continuous, non-decreasing process which increases only by jumps - we say that $\sigma$ is a pure jump process - and that the pairs $(t, x)$ induced by the times and sizes of the jumps are distributed as the atoms of a Poisson random measure on $[0, 1] \times (0,\infty)$ with intensity $(2 \pi x^3)^{-1/2} \d t \d x$.
Denote by $(P_i)_{i \ge 1}$ a size-biased permutation of the sequence $(J_i/\sum_k J_k)_{i \ge 1}$. Pitman \cite{Pitman-Poisson_Kingman_partitions-test} gives an inductive construction of a regular conditional distribution for $(P_i)_{i \ge 1}$ given $\{\sum_k J_k = z\}$ for arbitrary $z > 0$. The latter determines the conditional distribution of $(J_i/\sum_k J_k)_{i \ge 1}$ given $\{\sum_k J_k = z\}$ called Poisson-Kingman distribution. Descriptions of finite-dimensional distributions can be found in Perman \cite{Perman-Order_statistics_for_jumps_of_normalised_subordinators} or in Pitman and Yor \cite{Pitman_Yor-The_two_parameter_Poisson_Dirichlet_distribution_derived_from_a_stable_subordinator}. Our purpose here is to check that these distributions depend continuously on the variable $z$.

\begin{Pro}\label{pro-continuite_sauts_pont_stable}
The conditional distribution of the ranked jump-sizes $(J_i)_{i \ge 1}$ given $\{\sigma(1) = z\}$ is continuous in $z$.
\end{Pro}

\begin{proof}
In the recent work of Chaumont and Uribe Bravo \cite{Chaumont_Uribe_Bravo-Markovian_bridges_weak_continuity_and_pathwise_constructions}, sufficient conditions on the distribution of a Markov process $(X_t, t \ge 0)$ in a quite general metric space are given in order to make sense of a conditioned version of $(X_s, 0 \le s \le t)$ given $\{X_0 = x$ and $X_t = y\}$. The latter is called Markovian bridge from $x$ to $y$ of length $t$ and its law is denoted by $\mathbf{P}_{x, y}^t$. The process $\sigma$ fulfills the framework of their Theorem 1 and Corollary 1, it follows that the bridge laws $\mathbf{P}_{0, z}^1$ are well defined and continuous in $z$ for the Skorohod topology. Thanks to Skorohod's representation Theorem, the claim thus reduces to the deterministic result below.
\end{proof}

Let $f, f_1, f_2, \dots$ be functions defined from $[0, 1]$ to $[0, \infty)$ which are non-decreasing, right-continuous and null at $0$. Denote by $j_1 \ge j_2 \ge \dots \ge 0$ the ranked sizes of the jumps of $f$ and respectively, $j^{(n)}_1 \ge j^{(n)}_2 \ge \dots \ge 0$ that of $f_n$ for every $n \ge 1$.

\begin{Lem}\label{lem-convergence_skorohod_et_sauts}
Suppose that $f_n$ converges to $f$ for the Skorohod topology. Then
\begin{enumerate}
\item For any integer $N$, $(j^{(n)}_1, \dots, j^{(n)}_N)$ converges to $(j_1, \dots, j_N)$ in $\R^N$.
\item If $f$ is a pure jump function, then $(j^{(n)}_k)_{k \ge 1}$ converges to $(j_k)_{k \ge 1}$.
\end{enumerate}
\end{Lem}

\begin{proof}
For the first claim, suppose first that $f$ has infinitely many jumps. We may, and do, assume that $N$ is such that $j_N > j_{N+1}$. For any $t$, denote by $\Delta f(t) \coloneqq f(t)-f(t^-)$ the size of the jump made by $f$ at time $t$ and similarly $\Delta f_n$ for every $n \ge 1$. Upon changing the time scale using a sequence of increasing homeomorphisms from $[0, 1]$ onto itself which converges uniformly to the identity, we may assume that $f_n$ converges to $f$ uniformly. This does not affect the jump-sizes of $f_n$. Then $\Delta f_n(t)$ converges to $\Delta f(t)$ for every $t$ and $(j_1, \dots, j_N)$ are limits of $N$ jumps of $f_n$. Moreover, these jumps are $(j^{(n)}_1, \dots, j^{(n)}_N)$ for $n$ large enough since, for any $\eps \in (0, j_N-j_{N+1})$, for any $n$ large enough, as $f_n$ converges to $f$ uniformly, it admits no other jump larger than $j_{N+1}+\eps/2 < j_N-\eps/2$. If $f$ has only finitely many jumps, say, $N$, this reasoning yields the convergences $(j^{(n)}_1, \dots, j^{(n)}_N) \to (j_1, \dots, j_N)$ and $j^{(n)}_{k} \to 0$ for any $k \ge N+1$.

For the second claim, we write for any integer $N$ fixed,
\begin{equation*}
\sum_{k=1}^\infty \vert j^{(n)}_k - j_k\vert
\le \sum_{k=1}^N \vert j^{(n)}_k - j_k\vert + \sum_{k=N+1}^\infty j_k + \sum_{k=N+1}^\infty j^{(n)}_k.
\end{equation*}
As $n \to \infty$, the first term tends to 0 from (i). Let $\eps > 0$ and fix $N$ such that $\sum_{k=N+1}^\infty j_k < \eps$. Since $f$ is a pure jump function, we have $f(1) = \sum_{k=1}^\infty j_k$ and so $\sum_{k=1}^N j_k \ge f(1) - \eps$. Finally, since $\lim_{n \to \infty} f_n(1)=f(1)$, we conclude that $\sum_{k=N+1}^\infty j^{(n)}_k \le f_n(1) - \sum_{k=1}^N j^{(n)}_k \le 2\eps$ for $n$ large enough.
\end{proof}


\section{Asymptotic size of the burnt subtrees in the critical regime}\label{section3}

Fix $c \in (0, \infty)$ and consider the critical regime $p_n \sim c n^{-1/2}$ of the fire dynamics on $\tn$. Let $\kappa_n$ be the number of burnt subtrees, $\b_{n,1}, \dots, \b_{n,\kappa_n}$ their respective size, listed in order of appearance, and finally $\b_{n,1}^* \ge \dots \ge \b_{n,\kappa_n}^*$ a non-increasing rearrangement of the latter. We can now state the main result of this section.

\begin{Thm}\label{thm-crit}
For all continuous and bounded maps $f: (0, 1) \to \R$ and $F: \ell^1(\R) \to \R$, we have
\begin{multline*}
\lim_{n \to \infty} \E\bigg[f\bigg(\frac{I_n}{n}\bigg) F\bigg(\frac{\b_{n,1}^*}{n}, \dots, \frac{\b_{n,\kappa_n}^*}{n}\bigg)\bigg]
\\
= \int_0^1 f(x) \E\bigg[F\bigg(\frac{(1 - x) J_1}{\sigma(1)}, \frac{(1 - x) J_2}{\sigma(1)}, \dots\bigg) \,\bigg\vert\, \sigma(1) = \frac{1 - x}{c^2 x^2}\bigg] \P(D(c) \in \d x),
\end{multline*}
where $\sigma$ is a subordinator distributed as \eqref{eq-definition_subordinateur_S} and $\P(D(c) \in \d x)$ is defined in \eqref{eq-distribution_D_infini}.
\end{Thm}

Note that, taking $F \equiv 1$, this recovers the result \ref{thm-transition_de_phase_crit} in the introduction; moreover, since $\sum_i J_i = \sigma(1)$, it strengthens \ref{thm-transition_de_phase_crit} by giving the decomposition of the burnt forest conditionally given its total size.

The proof is divided in two parts. As discussed in the introduction, we view the fire dynamics on $\tn$ as a mark process on the associated cut-tree $\Cutn$, which translates the vector $n^{-1} (I_n, \b_{n,1}^*, \dots, \b_{n,\kappa_n}^*)$ into the proportion of leaves of the trees in the forest obtained by logging $\Cutn$ at the marks. We prove that the marked tree $\Cutn$, properly rescaled, converges to the CRT endowed with a certain point process; it follows that the previous vector converges to the masses of the trees in the forest obtained by logging the CRT at the atoms of the point process. We then study the distribution of the latter. As direct computations with the CRT seem rather complicated, we approximate the marked CRT by a Galton-Watson tree with Poisson(1) offspring distribution conditioned to have $n$ vertices and endowed with the same mark process as the cut-tree $\Cutn$. We refer to Aldous \cite{Aldous-The_continuum_random_tree_3} and Aldous and Pitman \cite{Aldous_Pitman-Standard_additive_coalescent} for prerequisites about the CRT, its logging by a Poisson point process and convergence of conditioned Galton-Watson trees.

\subsection{Binary cut-tree, fire dynamics and mark process}\label{section31}

Given a tree $T_n$ on a set of $n$ labeled vertices, say $[n] = \{1, \dots, n\}$, we build inductively its cut-tree $\Cut(T_n)$, which is a random rooted binary tree with $n$ leaves. Each vertex of $\Cut(T_n)$ corresponds to a subset (or \emph{block}) of $[n]$, the root of $\Cut(T_n)$ is the whole set $[n]$ and its leaves are the singletons $\{1\}, \dots, \{n\}$. We remove successively the edges of $T_n$ in a random uniform order; at each step, a subtree of $T_n$ with set of vertices, say, $B$, falls into two subtrees with set of vertices, say, $B'$ and $B''$ respectively; in $\Cut(T_n)$, $B'$ and $B''$ are the two offsprings of $B$. Notice that, by construction, the set of leaves of the subtree of $\Cut(T_n)$ generated by some block coincides with this block. See \autoref{fig-definition_cut_tree} for an illustration.

\begin{figure}[ht] \centering
\begin{tabular}{ccc}
\begin{tikzpicture}
\draw
	(-0.2,1.82)-- (0.78,2.02)		
	(0,0)-- (-0.5,0.87)			
	(0.66,3.01)-- (0.02,3.78)		
	(-0.5,0.87)-- (-0.97,-0.01)		
	(-0.2,1.82)-- (-0.43,2.88)		
	(1.43,3.65)-- (0.66,3.01)		
	(1,0)-- (0,0)				
	(0.78,2.02)-- (1.36,1.2)		
	(-0.5,0.87)-- (-0.2,1.82)		
	(0.66,3.01)-- (0.78,2.02)		
;
\begin{small}
\draw
	(0.78,2.02)	node (a) [circle, fill=white, draw]	{a}
	(-0.97,-0.01)	node (b) [circle, fill=white, draw]	{b}
	(-0.43,2.88)	node (c) [circle, fill=white, draw]	{c}	
	(1.36,1.2)		node (d) [circle, fill=white, draw]	{d}
	(0.66,3.01)	node (e) [circle, fill=white, draw]	{e}
	(0,0)			node (f) [circle, fill=white, draw]	{f}
	(1,0)			node (g) [circle, fill=white, draw]	{g}
	(0.02,3.78)	node (h) [circle, fill=white, draw]	{h}
	(-0.5,0.87)		node (i) [circle, fill=white, draw]	{i}
	(-0.2,1.82)		node (j) [circle, fill=white, draw]	{j}
	(1.43,3.65)	node (k) [circle, fill=white, draw]	{k}
;
\end{small}
\begin{footnotesize}
\draw
	(0.33,1.70)	node		{1}
	(-0.12,0.55)	node		{2}
	(0.44,3.53)	node		{3}
	(-0.86,0.54)	node		{4}
	(-0.17,2.4)		node		{5}
	(1,3.52)		node		{6}
	(0.53,0.17)	node		{7}
	(0.9,1.55)		node		{8}
	(-0.51,1.36)	node		{9}
	(0.5,2.55)		node		{10}
;
\end{footnotesize}
\end{tikzpicture}
&
&
%
\begin{tikzpicture}
[level 1/.style={sibling distance=4 cm},
level 2/.style={sibling distance=2 cm},
level 3/.style={sibling distance=1 cm},
level 4/.style={sibling distance=1 cm},
level 5/.style={sibling distance=1 cm},
level 6/.style={sibling distance=1 cm}]
\node[individu]{abcdefghijk} [grow'=up,level distance=.75cm]
	child{node[individu]{bcfgij}
		child{node[individu]{fg}
			child{node[individu]{g}}
			child{node[individu]{f}}
		}
		child{node[individu]{bcij}
			child{node[individu]{cij}
				child{node[individu]{ij}
					child{node[individu]{j}}
					child{node[individu]{i}}
				}
				child{node[individu]{c}}
			}
			child{node[individu]{b}}
		}
	}
	child{node[individu]{adehk}
		child{node[individu]{h}}
		child{node[individu]{adek}
			child{node[individu]{k}}
			child{node[individu]{ade}
				child{node[individu]{d}}
				child{node[individu]{ae}
					child{node[individu]{e}}
					child{node[individu]{a}}
				}
			}
		}
	}
;
\end{tikzpicture}
\end{tabular}
\caption{A tree with the order of cuts on the left and the corresponding cut-tree on the right.}
\label{fig-definition_cut_tree}
\end{figure}

Let $\tn$ be a Cayley tree with $n$ vertices and let $\Cutn$ be the associated cut-tree. We can encode the fire dynamics on $\tn$ as a mark process on the vertices of $\Cutn$. It is convenient to see fireproof edges of $\tn$ as deleted, then when an edge is set on fire, the whole subtree that contains it burns instantly and we mark the corresponding block of $\Cutn$. The leaves of $\Cutn$ cannot be marked as they correspond to singletons in $\tn$. Note that if a block of $\Cutn$ is marked, its descendants are never marked because the edges of the corresponding subtree of $\tn$ are no longer subject to the dynamics. The marked blocks of $\Cutn$ are exactly the burnt components of $\tn$ and, as we noticed, their size is the number of leaves of the subtree of $\Cutn$ they generate; see \autoref{fig-foret_et_arbre_binaire_marque} for an illustration.

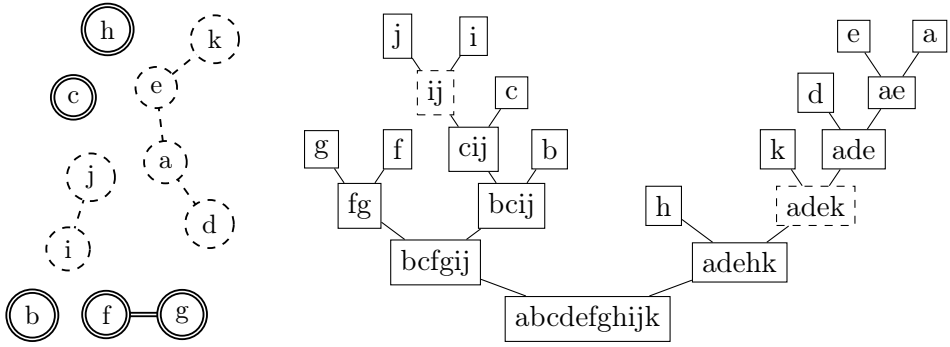
\begin{figure}[ht] \centering
\begin{tabular}{ccc}
\begin{tikzpicture}
\draw [line width=.7pt, double]
	(1,0)-- (0,0)				
;
\draw [line width=.7pt, dashed]
	(1.43,3.65)-- (0.66,3.01)		
	(0.78,2.02)-- (1.36,1.2)		
	(-0.5,0.87)-- (-0.2,1.82)		
	(0.66,3.01)-- (0.78,2.02)		
;
\begin{small}
\draw
	(0.78,2.02)	node (a) [circle, line width=.7pt, fill=white, draw, dashed]	{a}
	(-0.97,-0.01)	node (b) [circle, line width=.7pt, double, fill=white, draw]	{b}
	(-0.43,2.88)	node (c) [circle, line width=.7pt, double, fill=white, draw]	{c}	
	(1.36,1.2)		node (d) [circle, line width=.7pt, fill=white, draw, dashed]	{d}
	(0.66,3.01)	node (e) [circle, line width=.7pt, fill=white, draw, dashed]	{e}
	(0,0)			node (f) [circle, line width=.7pt, double, fill=white, draw]	{f}
	(1,0)			node (g) [circle, line width=.7pt, double, fill=white, draw]	{g}
	(0.02,3.78)	node (h) [circle, line width=.7pt, double, fill=white, draw]	{h}
	(-0.5,0.87)		node (i) [circle, line width=.7pt, fill=white, draw, dashed]	{i}
	(-0.2,1.82)		node (j) [circle, line width=.7pt, fill=white, draw, dashed]	{j}
	(1.43,3.65)	node (k) [circle, line width=.7pt, fill=white, draw, dashed]	{k}
;
\end{small}
\end{tikzpicture}
&
&
%
\begin{tikzpicture}
[
level 1/.style={sibling distance=4 cm},
level 2/.style={sibling distance=2 cm},
level 3/.style={sibling distance=1 cm},
level 4/.style={sibling distance=1 cm},
level 5/.style={sibling distance=1 cm},
level 6/.style={sibling distance=1 cm}
]
\node[individu]{abcdefghijk} [grow'=up,level distance=.75cm]
	child{node[individu]{bcfgij}
		child{node[individu]{fg}
			child{node[individu]{g}}
			child{node[individu]{f}}
		}
		child{node[individu]{bcij}
			child{node[individu]{cij}
				child{node[individu=dashed]{ij}
					child{node[individu]{j}}
					child{node[individu]{i}}
				}
				child{node[individu]{c}}
			}
			child{node[individu]{b}}
		}
	}
	child{node[individu]{adehk}
		child{node[individu]{h}}
		child{node[individu=dashed]{adek}
			child{node[individu]{k}}
			child{node[individu]{ade}
				child{node[individu]{d}}
				child{node[individu]{ae}
					child{node[individu]{e}}
					child{node[individu]{a}}
				}
			}
		}
	}
;
\end{tikzpicture}
\end{tabular}
\caption{The forest after the dynamics and the corresponding marked cut-tree.}
\label{fig-foret_et_arbre_binaire_marque}
\end{figure}

Given the cut-tree $\Cutn$, the mark process can be constructed as follows: at each generation, each internal (i.e. non-singleton) block is marked independently of the others with probability $p_n$ provided that none of its ancestor has been marked, and not marked otherwise. This is equivalent to the following two-steps procedure: mark first every internal block independently with probability $p_n$, then along each branch, keep only the closest mark to the root and erase the other marks. We will refer to this procedure as the \emph{marking-erasing process} associated with the point process which marks each internal block independently with probability $p_n$.

\subsection{Convergence of marked trees}\label{section32}

It will be more convenient to mark $\Cutn$ on its edges rather than on its vertices so we shift the marks defined above from a vertex to the edge that connects it to its parent: denote by $\varphi_n'$ the mark process which marks each edge of $\Cutn$ which is not adjacent to a leaf on its mid-point independently with probability $p_n$, by $\varphi_n$ the associated marking-erasing process and by $\# (\Cutn, \varphi_n)$ the vector whose entries count the number of leaves of each tree in the forest obtained by logging $\Cutn$ at the marks of $\varphi_n$, the root-component first, and the next in non-increasing order.

Let $\CRT$ be a rooted Brownian CRT, $\mu$ its uniform probability "mass" measure on leaves and the usual distance $d$. The distance induces a "length" measure $\ell$, which is the unique $\sigma$-finite measure assigning measure $d(x,y)$ to the geodesic path between $x$ and $y$ in $\CRT$. Denote by $\Phi'$ a Poisson point process with intensity $c \ell(\cdot)$ on the skeleton of $\CRT$, $\Phi$ the associated marking-erasing process and $\#(\CRT, \Phi)$ the vector whose entries count the mass of each tree in the forest obtained by logging $\CRT$ at the atoms of $\Phi$, again the root-component first, and the next in non-increasing order.

\begin{Lem}\label{lem-convergence_marked_trees_marked_CRT}
The vector $n^{-1} \# (\Cutn, \varphi_n)$ converges in distribution to $\# (\CRT, \Phi)$ for the $\ell^1$ topology.
\end{Lem}

We endow the tree $\Cutn$ with the uniform distribution on leaves $\mu_n$ and the metric $d_n$ given by the graph distance rescaled by a factor $n^{-1/2}$. For every integer $k \ge 1$, denote by $\Red_n(k)$ the smallest connected subset of $\Cutn$ which contains the root $[n]$ and $k$ i.i.d. leaves chosen according to $\mu_n$; we call $\Red_n(k)$ the tree $\Cutn$ reduced to those leaves. Denote similarly by $\Red(k)$, the CRT reduced to $k$ i.i.d. elements picked according to $\mu$. We see the reduced trees as finite rooted metric spaces; the proof of Lemma~1 in Bertoin \cite{Bertoin-Fires_on_trees} shows that for every $k \ge 1$ fixed,
\begin{equation}\label{eq-convergence_arbres_reduits}
\lim_{n \to \infty} \Red_n(k) = \Red(k)
\qquad\text{in distribution}
\end{equation}
in the sense of Gromov-Hausdorff. This is equivalent to the convergence of the rooted metric measure spaces $(\Cutn, d_n, \mu_n)$ to $(\CRT, d, \mu)$ in distribution for the so-called Gromov-Prokhorov topology.

\begin{proof}
Since the scaling factor of $\Cutn$ corresponds to $p_n$, we may, and do, extend \eqref{eq-convergence_arbres_reduits} to the joint convergence of $\Red_n(k)$ and the trace of $\varphi_n'$ on its edges to $\Red(k)$ endowed with a Poisson point process with rate $c$ per unit length on its edges. The same convergence holds when considering the marking-erasing processes; subsequently, for every $k \ge 1$,
\begin{equation*}
\lim_{n \to \infty} \# \Red_n(k, \varphi_n) = \# \Red(k, \Phi)
\end{equation*}
in distribution, where $\# \Red_n(k, \varphi_n)$ denotes the vector whose entries count the number of leaves of each tree in the forest obtained by logging $\Red_n(k)$ at the marks induced by $\varphi_n$, the root-component first, and the next in non-increasing order, and similarly for $\# \Red(k, \Phi)$. Since
\begin{equation*}
\lim_{k \to \infty} k^{-1} \# \Red(k, \Phi) = \# (\CRT, \Phi),
\end{equation*}
it follows from a diagonal argument that for $k_n \to \infty$ sufficiently slowly as $n \to \infty$,
\begin{equation*}
\lim_{n \to \infty} k_n^{-1} \# \Red_n(k_n, \varphi_n) = \# (\CRT, \Phi)
\end{equation*}
in distribution. Adapting Lemma 11 of Aldous and Pitman~\cite{Aldous_Pitman-Standard_additive_coalescent} to uniform sampling of leaves instead of vertices, this finally yields
\begin{equation*}
\lim_{n \to \infty} n^{-1} \# (\Cutn, \varphi_n) = \#(\CRT, \Phi)
\end{equation*}
in distribution.
\end{proof}

A consequence of this lemma is the following result, which is the first step in the proof of \autoref{thm-crit}. Recall that $I_n$ stands for the total number of fireproof vertices of the Cayley tree $\tn$, $\kappa_n$ for the number of burnt components and $\b_{n,1}^* \ge \dots \ge \b_{n,\kappa_n}^*$ for their respective size, ranked in non-increasing order.

\begin{Pro}\label{pro-convergence_taille_des_arbres_vers_masses_du_CRT}
We have
\begin{equation*}
\lim_{n \to \infty}
n^{-1} \big(I_n, \b_{n,1}^*, \dots, \b_{n,\kappa_n}^*) = \# (\CRT, \Phi)
\quad\text{in distribution.}
\end{equation*}
\end{Pro}

\begin{proof}
Recall that the connected components of $\Cutn$ that do not contain the root correspond to the burnt subtrees of $\tn$, whereas the root-component corresponds to the fireproof forest. Recall also that the number of leaves in $\Cutn$ of each component is the number of vertices of the corresponding in $\tn$. We then get the identity
\begin{equation*}
(I_n, \b_{n,1}^*, \dots, \b_{n,\kappa_n}^*) = \# (\Cutn, \varphi_n),
\end{equation*}
and the claim follows readily from \autoref{lem-convergence_marked_trees_marked_CRT}.
\end{proof}

To complete the proof of \autoref{thm-crit}, we need to identify the limiting distribution $\# (\CRT, \Phi)$. For this, we use a second discrete approximation of the latter. Denote by $\GWTn$ a Galton-Watson tree with Poisson(1) offspring distribution conditioned to have $n$ vertices and where labels are assigned to the vertices uniformly at random. It is known, see for instance Aldous \cite{Aldous-The_continuum_random_tree_3}, that $\GWTn$ is distributed as a uniform rooted Cayley tree with $n$ vertices and that, reduced to $k$ vertices picked uniformly at random and rescaled by a factor $n^{-1/2}$, it converges to the CRT reduced to $k$ leaves. We endow $\GWTn$ with the marking-erasing process $\psi_n$ associated with the process which marks each vertex independently with probability $p_n$. Adapting \autoref{lem-convergence_marked_trees_marked_CRT} to $\GWTn$ and the uniform probability on vertices, we get
\begin{equation}\label{eq-convergence_GWT_marque_vers_CRT_marque}
\lim_{n \to \infty} n^{-1} \# (\GWTn, \psi_n) = \# (\CRT, \Phi)
\quad\text{in distribution.}
\end{equation}
where $\# (\GWTn, \psi_n)$ stands here for the number of vertices of each component, the root-component first, and the next in non-increasing order. We now study the asymptotic behavior of this vector in order to show that the right-hand side above is the limit in \autoref{thm-crit}.

\subsection{Asymptotic behavior of the size of the burnt blocks}\label{section33}

Denote by $C_{n,0}$ the size of the connected component of $\GWTn$ that contains the root, $M_n$ the number of marks, and $C_{n,1}^* \ge \dots \ge C_{n,M_n}^*$ the respective sizes of the other connected components, listed in non-increasing order.

\begin{Pro}\label{pro-convergence_tailles_des_composantes_GWT_vers_masses_du_CRT}
For all continuous and bounded maps $f \colon (0,1) \to \R$ and $F \colon \ell^1(\R) \to \R$,
\begin{multline*}
\lim_{n \to \infty} \E\bigg[f\bigg(\frac{C_{n,0}}{n}\bigg) F\bigg(\frac{C_{n,1}^*}{n}, \dots, \frac{C_{n,M_n}^*}{n}\bigg)\bigg]
\\
= \int_0^1 f(x) \E\bigg[F\bigg(\frac{(1 - x) J_1}{\sigma(1)}, \frac{(1 - x) J_2}{\sigma(1)}, \dots\bigg) \,\bigg\vert\, \sigma(1) = \frac{1 - x}{c^2 x^2}\bigg] \P(D(c) \in \d x),
\end{multline*}
where $\sigma$ is a subordinator distributed as \eqref{eq-definition_subordinateur_S} and $\P(D(c) \in \d x)$ is defined in \eqref{eq-distribution_D_infini}.
\end{Pro}

Before proving this result, notice first that \autoref{thm-crit} is a direct consequence of Propositions~\ref{pro-convergence_taille_des_arbres_vers_masses_du_CRT} and \ref{pro-convergence_tailles_des_composantes_GWT_vers_masses_du_CRT} and the convergence \eqref{eq-convergence_GWT_marque_vers_CRT_marque}.

\begin{proof}[Proof of \autoref{thm-crit}]
Let $f \colon (0,1) \to \R$ and $F \colon \ell^1(\R) \to \R$ be two continuous and bounded maps. From \autoref{pro-convergence_taille_des_arbres_vers_masses_du_CRT} and \eqref{eq-convergence_GWT_marque_vers_CRT_marque}, the sequences
\begin{equation*}
\E\bigg[f\bigg(\frac{I_n}{n}\bigg) F\bigg(\frac{\b_{n,1}^*}{n}, \dots, \frac{\b_{n,\kappa_n}^*}{n}\bigg)\bigg]
\quad\text{and}\quad
\E\bigg[f\bigg(\frac{C_{n,0}}{n}\bigg) F\bigg(\frac{C_{n,1}^*}{n}, \dots, \frac{C_{n,M_n}^*}{n}\bigg)\bigg]
\end{equation*}
both converge to the same limit as $n \to \infty$ and \autoref{pro-convergence_tailles_des_composantes_GWT_vers_masses_du_CRT} gives the expression of the latter, which is claimed in \autoref{thm-crit}.
\end{proof}

It remains to prove \autoref{pro-convergence_tailles_des_composantes_GWT_vers_masses_du_CRT}. For any positive real number $z$, we define the Borel distribution with parameter $z$, which is the law of the size of a Galton-Watson tree with Poisson($z$) offspring distribution:
\begin{equation*}
\P(\text{Borel}(z) = n) = \frac{1}{n!} \ex^{-n z} (n z)^{n-1},\qquad n \ge 1.
\end{equation*}
We also define for any integer $k$, the Borel-Tanner distribution with parameter $k$ as the sum of $k$ i.i.d. Borel(1) variables:
\begin{equation*}
\P(\text{Borel-Tanner}(k) = n) = \frac{k}{(n-k)!} \ex^{-n} n^{n-k-1},\qquad n \ge k.
\end{equation*}
Borel and Borel-Tanner distributions appear in our context as the sizes of the connected components of $\GWTn$.

\begin{Lem}\label{lem-petits_GW}
For any integers $x, y$ with $x+y \le n$, conditionally on the event $\{C_{n,0} = x, M_n = y\}$, the vector $(C_{n,1}^*, \dots, C_{n,y}^*)$ is distributed as a non-increasing rearrangement of $y$ i.i.d. {\rm Borel}$(1)$ random variables conditioned to have sum $n - x$.
\end{Lem}

\begin{proof}
We explicitly write the condition for the size of the tree. Let $\GWT$ be a Galton-Watson tree with Poisson(1) offspring distribution; we endow it with the marking-erasing process associated with the process which marks each vertex independently with probability $p_n$. Denote by $\tilde{M}_n$ the number of marks, $\tilde{C}_{n,0}$ the size of the root-component and, conditionally on $\{\tilde{M}_n = y\}$, $\tilde{C}_{n,1}^* \ge \dots \ge \tilde{C}_{n,y}^*$ the ranked sizes of the other components. Note that on the event $\{\tilde{M}_n = y\}$, we have $\vert\GWT\vert = \tilde{C}_{n,0} + \tilde{C}_{n,1}^* + \dots + \tilde{C}_{n,y}^*$.

Condition on the event $\{\tilde{M}_n = y\}$; it is known that the subtrees of $\GWT$ generated by the $y$ atoms of the point process are independent Galton-Watson trees with Poisson(1) offspring distribution, independent of $\tilde{C}_{n,0}$. Hence, on the event $\{\tilde{M}_n = y\}$, $\tilde{C}_{n,1}^*, \dots, \tilde{C}_{n,y}^*$ are i.i.d. Borel(1) random variables, listed in non-increasing order and independent of $\tilde{C}_{n,0}$. Further, on the event $\{\vert\GWT\vert = n, \tilde{M}_n = y, \tilde{C}_{n,0} = x\}$, $\tilde{C}_{n,1}^*, \dots, \tilde{C}_{n,y}^*$ are conditioned to have sum $n - x$.
\end{proof}

The Borel(1) distribution belongs to the domain of attraction of the stable law of index $1/2$. A consequence tailored for our need is the following: let $(\beta_i)_{i \ge 1}$ be i.i.d. Borel$(1)$ random variables, and for any $k \ge 1$, denote by $\beta^*_1 \ge \dots \ge \beta^*_k$ the order statistics of the first $k$ elements of the latter. Let also $\sigma$ be a subordinator distributed as \eqref{eq-definition_subordinateur_S} and $J_1 \ge J_2 \ge \dots \ge 0$ the ranked sizes of its jumps made during the time interval $[0, 1]$.

\begin{Lem}\label{lem-convergence_somme_borel}
Let $\lambda, \nu >0$ and two sequences of integers $k_n$ and $a_n$ such that $\lim_{n \to \infty} n^{-1/2} k_n = \lambda$ and $\lim_{n \to \infty} n^{-1} a_n = \nu$. Then
\begin{equation*}
\lim_{n \to \infty} \Bigg(\Bigg(\frac{1}{n} \sum_{i=1}^{\lfloor \sqrt{n} t\rfloor \wedge k_n} \beta_i\,,\, t \ge 0\Bigg) \,\Bigg\vert\, \sum_{i=1}^{k_n} \beta_i = a_n \Bigg)
= \big((\sigma(t \wedge \lambda), t \ge 0) \,\big\vert\, \sigma(\lambda) = \nu\big)
\end{equation*}
in distribution for the Skorohod topology. As a consequence, the convergence of the ranked jumps holds for the $\ell^1$ topology:
\begin{equation*}
\lim_{n \to \infty} \Bigg(\Bigg(\frac{\beta^*_1}{n}, \dots, \frac{\beta^*_{k_n}}{n}\Bigg) \,\Bigg\vert\, \sum_{i=1}^{k_n} \beta_i = a_n \Bigg) = \Bigg(\frac{\nu J_1}{\sigma(1)}, \frac{\nu J_2}{\sigma(1)}, \dots \,\Bigg\vert\, \sigma(1) = \frac{\nu}{\lambda^2} \Bigg)
\end{equation*}
in distribution.
\end{Lem}

\begin{proof}
The first convergence is the result stated in Lemma 11 of Aldous and Pitman \cite{Aldous_Pitman-Brownian_bridge_asymptotics_for_random_mappings}. The second then follows from the continuity obtained in \autoref{lem-convergence_skorohod_et_sauts}.
\end{proof}

We apply this convergence to the random sequences $M_n$ and $n - C_{n,0}$ instead of $k_n$ and $a_n$. They fulfill the assumptions of \autoref{lem-convergence_somme_borel} as it is shown in the following lemma that we prove in the next subsection.

\begin{Lem}\label{lem-convergence_taille_composante_contenant_la_racine_et_nombre_de_marques_du_GWT}
Let $D(c)$ be a random variable distributed as \eqref{eq-distribution_D_infini}. Then
\begin{equation*}
\lim_{n \to \infty} \bigg(\frac{C_{n,0}}{n}, \frac{M_n}{\sqrt{n}}\bigg) = (D(c), c D(c))
\quad\text{in distribution.}
\end{equation*}
\end{Lem}

In order to go from deterministic sequences to random sequences, we also use the following elementary result (see Carath\'{e}odory \cite{Caratheodory-Theory_of_functions_of_a_complex_variable}, Part Four, Chapter I). Let $\mathbb{X}$ and $\mathbb{Y}$ be metric spaces and $f, f_1, f_2, \dots$ be functions defined from $\mathbb{X}$ to $\mathbb{Y}$. We say that $f_n$ converges continuously to $f$ if for any $x, x_1, x_2,\dots \in \mathbb{X}$ such that $\lim_{n \to \infty} x_n = x$ in $\mathbb{X}$, we have $\lim_{n \to \infty} f_n(x_n) = f(x)$ in $\mathbb{Y}$. Then $f_n$ converges continuously to $f$ if and only if $f$ is continuous and $f_n$ converges to $f$ uniformly on compact sets.

\begin{proof}[Proof of \autoref{pro-convergence_tailles_des_composantes_GWT_vers_masses_du_CRT}]
Let $f \colon (0,1) \to \R$ and $F \colon \ell^1(\R) \to \R$ be two continuous and bounded maps. With the notations of \autoref{lem-convergence_somme_borel}, define for any $(u, v) \in (0,1)\times(0,\infty)$
\begin{equation*}
\Upsilon_n(u, v) \coloneqq f(u) \E\bigg[F\bigg(\frac{\beta^*_1}{n}, \dots, \frac{\beta^*_{\lfloor \sqrt{n} v \rfloor}}{n}\bigg) \,\bigg\vert\, \sum_{i=1}^{\lfloor \sqrt{n} v \rfloor} \beta_i = n - \lfloor n u\rfloor \bigg],
\end{equation*}
and
\begin{equation*}
\Upsilon(u, v) \coloneqq f(u) \E\bigg[F\bigg(\frac{(1 - u) J_1}{\sigma(1)}, \frac{(1 - u) J_2}{\sigma(1)}, \dots\bigg) \,\bigg\vert\, \sigma(1) = \frac{1 - u}{v^2}\bigg].
\end{equation*}
Then \autoref{lem-convergence_somme_borel} states that $\Upsilon_n(u_n, v_n) \to \Upsilon(u, v)$ whenever $(u_n, v_n) \to (u, v)$. On the one hand, $\E[\Upsilon(D(c), c D(c))]$ is the limit claimed in \autoref{pro-convergence_tailles_des_composantes_GWT_vers_masses_du_CRT} and, from \autoref{lem-petits_GW}, the $C_{n,i}^*$'s are, conditionally given $C_{n,0}$ and $M_n$, distributed as ranked i.i.d. Borel(1) random variables conditioned to have sum $n-C_{n,0}$. Then we also have
\begin{equation*}
\E\bigg[\Upsilon_n\bigg(\frac{C_{n,0}}{n}, \frac{M_n}{\sqrt{n}}\bigg)\bigg] = \E\bigg[f\bigg(\frac{C_{n,0}}{n}\bigg) F\bigg(\frac{C_{n,1}^*}{n}, \dots, \frac{C_{n,M_n}^*}{n}\bigg)\bigg].
\end{equation*}
On the other hand, from the discussion above, $\Upsilon$ is continuous (which is also a consequence of \autoref{pro-continuite_sauts_pont_stable}) and $\Upsilon_n \to \Upsilon$ uniformly on compact sets. We then write
\begin{multline*}
\bigg\vert	\E\bigg[\Upsilon_n\bigg(\frac{C_{n,0}}{n}, \frac{M_n}{\sqrt{n}}\bigg)\bigg]
		- \E\big[\Upsilon(D(c), c D(c))\big]
\bigg\vert
\\
\le		\E\bigg[\bigg\vert\Upsilon_n\bigg(\frac{C_{n,0}}{n}, \frac{M_n}{\sqrt{n}}\bigg) - \Upsilon\bigg(\frac{C_{n,0}}{n}, \frac{M_n}{\sqrt{n}}\bigg)\bigg\vert\bigg]
+		\bigg\vert \E\bigg[\Upsilon\bigg(\frac{C_{n,0}}{n}, \frac{M_n}{\sqrt{n}}\bigg)\bigg] - \E\big[\Upsilon(D(c), c D(c))\big]\bigg\vert.
\end{multline*}
From \autoref{lem-convergence_taille_composante_contenant_la_racine_et_nombre_de_marques_du_GWT}, since $\Upsilon$ is continuous and bounded, the second term tends to $0$. Moreover $\Upsilon, \Upsilon_1, \Upsilon_2, \dots$ are uniformly bounded, say by $C > 0$, therefore the first term is bounded from above by
\begin{equation*}
\sup_{x \in K} \big\vert\Upsilon_n(x) - \Upsilon(x)\big\vert + 2C \P\bigg(\bigg(\frac{C_{n,0}}{n}, \frac{M_n}{\sqrt{n}}\bigg) \notin K\bigg),
\end{equation*}
for any compact $K$. The first term of the latter tends to $0$ for any $K$ and the second can be made arbitrary small as the sequence is tight.
\end{proof}

\subsection{Asymptotic behavior of the number of burnt blocks}\label{section34}

We finally prove \autoref{lem-convergence_taille_composante_contenant_la_racine_et_nombre_de_marques_du_GWT} which completes the proof of \autoref{pro-convergence_tailles_des_composantes_GWT_vers_masses_du_CRT} and thereby that of \autoref{thm-crit}. We use the two following observations: the convergence of the first marginal $n^{-1} C_{n,0}$ holds and the conditional distribution of $M_n$ given $C_{n,0}$ is known explicitly.

\begin{Lem}\label{lem-convergence_taille_composante_contenant_la_racine_du_GWT}
We have $\lim_{n \to \infty} n^{-1} C_{n,0} = D(c)$ in distribution.
\end{Lem}

\begin{proof}
Denote by $\mu(\xi_0)$ the mass of the root-component of the CRT after logging at the atoms of a Poisson point process with rate $c$ per unit length (here keeping only the closest atoms to the root does not matter). Then \eqref{eq-convergence_GWT_marque_vers_CRT_marque} yields
\begin{equation*}
\lim_{n \to \infty} n^{-1} C_{n,0} = \mu(\xi_0)
\quad\text{in distribution.}
\end{equation*}
The claim then follows from the identity $\mu(\xi_0) = D(c)$ in distribution stated in Corollary 5 of Aldous and Pitman \cite{Aldous_Pitman-Standard_additive_coalescent} since $\mu(\xi_0)$ here is $Y_1^*(c)$ there.
\end{proof}

\begin{Lem}\label{lem-distribution_taille_composante_contenant_la_racine_et_nombre_de_marques_du_GWT} 
For any $n \ge 2$, the pair $(C_{n,0}, M_n)$ is distributed as follows: for any integers $x, y$ such that $x+y \le n$,
\begin{equation*}
\P(C_{n,0} = x, M_n = y) = \frac{n! (x(1-p_n))^{x-1} (x p_n)^{y} (n-x)^{n-x-y-1}}{n^{n-1} x! (y-1)! (n - x - y)!}.
\end{equation*}
Then, on the event $\{C_{n,0} =x\}$, $M_n$ is distributed as $X_n + 1$ where $X_n$ is a binomial random variable with parameters $n-x-1$ and $(x p_n)/(n-x+x p_n)$.
\end{Lem}

\begin{proof}
For the first claim, as in the proof of \autoref{lem-petits_GW}, we explicitly write the condition on the size of the tree and work with a Galton-Watson tree with Poisson(1) offspring distribution $\GWT$:
\begin{equation*}
\P(C_{n,0} = x, M_n = y \,\vert\, \vert\GWT\vert = n )
= \frac{\P(C_{n,0} = x) \, \P(M_n = y \,\vert\, C_{n,0} = x) \, \P(\vert\GWT\vert = n \,\vert\, C_{n,0} = x, M_n = y)}{\P(\vert\GWT\vert = n)}.
\end{equation*}
We know that $\vert\GWT\vert$ is Borel(1) distributed. Moreover, the root-component of $\GWT$ is a Galton-Watson tree with Poisson$(1-p_n)$ offspring distribution, so that $C_{n,0}$ is Borel$(1-p_n)$ distributed, and on $\{C_{n,0} = x\}$, $M_n$ is the sum of $x$ i.i.d. Poisson$(p_n)$ random variables, so is Poisson$(x p_n)$ distributed. Finally, from \autoref{lem-petits_GW}, on $\{C_{n,0} = x, M_n = y\}$, $\vert\GWT\vert - x$ is the sum of $y$ i.i.d. Borel(1) random variables, i.e. is Borel-Tanner($y$) distributed. Putting the pieces together gives the first claim. For the second claim (with the implicit condition $\vert\GWT\vert = n$), we then directly compute
\begin{align*}
\P(M_n = y \,\vert\, C_{n,0} = x)
&=	\P(C_{n,0} = x, M_n = y) \bigg(\sum_{z = 1}^{n-x} \P(C_{n,0} = x, M_n = z) \bigg)^{-1}
\\
&=	\frac{(n - x - 1)!}{(y-1)! (n - x - y)!} \bigg(\frac{x p_n}{n - x + x p_n}\bigg)^{ y-1} \bigg(\frac{n-x}{n - x + x p_n}\bigg)^{ n-x-y}
\\
&=	\P(X_n = y-1),
\end{align*}
where $X_n$ is the desired binomial random variable.
\end{proof}

We can now prove \autoref{lem-convergence_taille_composante_contenant_la_racine_et_nombre_de_marques_du_GWT}.

\begin{proof}[Proof of \autoref{lem-convergence_taille_composante_contenant_la_racine_et_nombre_de_marques_du_GWT}]
We aim to show that for any $s, t \ge 0$,
\begin{equation*}
\lim_{n \to \infty} \E\bigg[\exp\bigg(-s \frac{C_{n,0}}{n} - t \frac{M_n}{\sqrt{n}} \bigg)\bigg] = \E[\exp(-(s+c t) D(c) )].
\end{equation*}
From \autoref{lem-convergence_taille_composante_contenant_la_racine_du_GWT}, $n^{-1} C_{n,0} \to D(c)$ in distribution, it is thus sufficient to show
\begin{equation*}
\limsup_{n \to \infty}
\bigg\vert
	\E\bigg[\exp\bigg(-s \frac{C_{n,0}}{n} - t \frac{M_n}{\sqrt{n}} \bigg)\bigg]
	-
	\E\bigg[\exp\bigg(-(s+ct) \frac{C_{n,0}}{n}\bigg)\bigg]
\bigg\vert
= 0.
\end{equation*}
Let $\eps > 0$ and fix $\delta > 0$ such that for any $n$ large enough, $\P(C_{n,0} > \lfloor(1-\delta) n\rfloor) \le \eps$. We then reduce to show the above convergence on the event $\{C_{n,0} \le \lfloor(1-\delta) n\rfloor\}$. Using \autoref{lem-distribution_taille_composante_contenant_la_racine_et_nombre_de_marques_du_GWT}, we compute for any $1 \le x \le n-1$ and any $t \ge 0$,
\begin{equation*}
\E\big[\ex^{-t M_n} \,\big\vert\, C_{n,0} = x\big] = \E\big[\ex^{-t (X_n + 1)}\big] = \ex^{-t} \bigg(1 - \frac{x p_n (1 - \ex^{- t})}{n-x + x p_n}\bigg)^{n - x -1}.
\end{equation*}
Conditioning first on the value of $C_{n,0}$ and then averaging, we get
\begin{multline*}
\E\bigg[\exp\bigg(-s \frac{C_{n,0}}{n} - t \frac{M_n}{\sqrt{n}}\bigg) \ind{C_{n,0} \le \lfloor(1-\delta) n\rfloor}\bigg]
\\
= \sum_{x = 1}^{\lfloor(1-\delta) n\rfloor} \P\big(C_{n,0} = x\big) \exp\bigg(-\frac{s x}{n}\bigg) \exp\bigg(- \frac{t}{\sqrt{n}}\bigg) \bigg(1 - \frac{x p_n (1 - \ex^{-t/\sqrt{n}})}{n-x + x p_n}\bigg)^{n - x -1}.
\end{multline*}
Remark that, uniformly for $x \le \lfloor(1-\delta) n\rfloor$,
\begin{equation*}
\frac{x p_n (1 - \ex^{-t/\sqrt{n}})}{n-x + x p_n} = \frac{1}{n-x} \bigg(\frac{x c t}{n} + o(1)\bigg) \quad\text{as }n \to \infty.
\end{equation*}
As a consequence, as $n \to \infty$,
\begin{equation*}
\exp\bigg(- \frac{t}{\sqrt{n}}\bigg) \bigg(1 - \frac{x p_n (1 - \ex^{-t/\sqrt{n}})}{n-x + x p_n}\bigg)^{n - x -1}
= \exp\bigg(- \frac{x c t}{n}\bigg) (1 + o(1)),
\end{equation*}
uniformly for $x \le \lfloor(1-\delta) n\rfloor$. Finally, the difference
\begin{equation*}
\E\bigg[\exp\bigg(-s \frac{C_{n,0}}{n} - t \frac{M_n}{\sqrt{n}}\bigg) \ind{C_{n,0} \le \lfloor(1-\delta) n\rfloor}\bigg]
 - 
\E\bigg[\exp\bigg(-(s+ct) \frac{C_{n,0}}{n}\bigg) \ind{C_{n,0} \le \lfloor(1-\delta) n\rfloor}\bigg]
\end{equation*}
tends to 0 as $n \to \infty$, which completes the proof.
\end{proof}


\section{Asymptotic proportion of fireproof vertices in the subcritical regime}\label{section4}

We now consider the subcritical regime $p_n \gg n^{-1/2}$ of the dynamics on $\tn$. We prove the convergence of the total number of fireproof vertices $I_n$, rescaled by a factor $p_n^2$ (recall \autoref{thm-sous_crit}), and also the following result on the size of the largest fireproof component.

\begin{Pro}\label{pro-composante_geante_cas_sous_critique}
For any $\eps > 0$, with a probability converging to $1$ as $n \to \infty$, there exists at least one fireproof subtree larger than $n^{- \eps} p_n^{-2}$ but none larger than $\eps p_n^{-2}$. 
\end{Pro}

Let us sketch our approach to establish \autoref{thm-sous_crit}. We let the dynamics evolve until an edge is set on fire for the first time, denoting this random time by $\zeta_n \in \N \cup \{\infty\}$. The event $\{\zeta_n = \infty\}$ corresponds to the case where the whole tree is fireproof at the end. Conditionally on $\{\zeta_n = k\}$ with $k \in \N$, if we delete the $k-1$ first fireproof edges, we get a decomposition of $\tn$ into a forest of $k$ trees. Then we set on fire an edge of this forest uniformly at random and burn the whole subtree that contains the latter. The burnt subtree is therefore picked at random with a probability proportional to its number of edges. We then study the dynamics which continue independently on each of the $k-1$ other subtrees.

Let $\sigma$ be a subordinator distributed as \eqref{eq-definition_subordinateur_S} and $J_1 \ge J_2 \ge \dots \ge 0$ the sizes of its jumps made during the time interval $[0, 1]$. Let also $\e$ be an exponential random variable with parameter $1$ independent of $\sigma$. We shall see that $p_n \zeta_n$ converges to $\e$ in distribution and that the sequence of the sizes of the non-burnt subtrees at time $\zeta_n$, ranked in non-increasing order and rescaled by a factor $p_n^2$, converges in distribution to $(\e^2 J_k)_{k \ge 1}$ in $\ell^1$. Conditionally given $(\e^2 J_k)_{k \ge 1}$, we define a sequence $(X_k(\e))_{k \ge 1}$ of independent random variables sampled according to $\mu_{\e^2 J_k}$ respectively, where for every $x > 0$, $\mu_x$ is the probability measure given by
\begin{equation}\label{eq-generalisation_distribution_proportion_sites_ignifuges_dans_un_arbre_de_Cayley}
\mu_x(\d y) = \bigg(\frac{x^3}{2 \pi y (x - y)^3}\bigg)^{1/2} \exp\bigg(-\frac{x y}{2 (x - y)}\bigg) \d y,
\qquad 0 < y < x.
\end{equation}
Note that if $X$ is distributed as $\mu_x$, then $x^{-1} X$ is distributed as $D(x^{1/2})$, defined in \eqref{eq-distribution_D_infini}. Indeed, $\mu_x$ is the limit of the number of fireproof vertices in a subtree of asymptotic size $x$ (see \autoref{lem-generalistion_thm_convergence_densite_sites_ignifuges_cas_sous_critique} for a precise statement). Informally, summing over all subtrees, since the dynamics on each are independent, we get
\begin{equation}\label{eq-convergence_somme_arbres_brules_cas_sous_critique}
\lim_{n \to \infty} p_n^2 I_n = \sum_{k=1}^\infty X_k(\e)
\quad\text{in distribution.}
\end{equation}
\autoref{thm-sous_crit} finally follows from the identity
\begin{equation}\label{eq-identification_xi_e_gaussienne_au_carre}
\sum_{k=1}^\infty X_k(\e) = \Gauss^2
\quad\text{in distribution.}
\end{equation}
To derive the latter, note that, conditionally given $\e$, the sequence $(\e^2 J_k)_{k \ge 1}$ is distributed as the ranked atoms of a Poisson random measure on $(0, \infty)$ with intensity $\e (2 \pi x^3)^{-1/2} \d x$. Further, conditionally given $\e$, the sequence $(\e^2 J_k, X_k(\e))_{k \ge 1}$ is distributed as the atoms of a Poisson random measure on $(0, \infty)^2$ with intensity $\e (2 \pi x^3)^{-1/2} \d x \mu_x(\d y)$, ranked in the non-increasing order of the first coordinate. Therefore, conditioning first on $\e$, using Laplace formula and then averaging, we have for any $q > 0$,
\begin{equation*}
\E\bigg[\exp\bigg(-q \sum_{k=1}^\infty X_k(\e)\bigg)\bigg]
= \int_0^\infty \exp\bigg(- \int_{(0, \infty)^2} (1 - \ex^{-qy}) t (2 \pi x^3)^{-1/2} \d x \mu_x(\d y)\bigg) \ex^{-t} \d t.
\end{equation*}
Using the definition of $\mu_x$ and the change of variables $(x, y) \mapsto (y(x-y)^{-1/2}, y)$, we see that the right-hand side is equal to
\begin{equation*}
\int_0^\infty \exp\bigg(- t - t \int_0 ^\infty (1 - \ex^{-qy}) \ex^{- y/2} \frac{\d y}{\sqrt{2 \pi y^3}}\bigg) \d t.
\end{equation*}
We write
\begin{equation*}
(1 - \ex^{-qy}) \ex^{- y/2}
= \bigg(1 - \exp\bigg(- \frac{2q+1}{2} y\bigg)\bigg) - \bigg(1 - \exp\bigg(- \frac{y}{2}\bigg)\bigg);
\end{equation*}
since
\begin{equation*}
\int_0^\infty \bigg(1 - \exp\bigg( - \frac{z^2 y}{2}\bigg)\bigg) \frac{\d y}{\sqrt{2 \pi y^3}} = z
\quad\text{for any } z > 0,
\end{equation*}
we finally get,
\begin{equation*}
\E\bigg[\exp\bigg(-q \sum_{k=1}^\infty X_k(\e)\bigg)\bigg]
= \int_0^\infty \exp\big(- t \sqrt{2q+1}\big) \d t
= (2q+1)^{-1/2}
= \E\big[\exp\big(- q \Gauss^2\big)\big].
\end{equation*}

In the rest of this section, we first prove the convergence of the sequence of the sizes of the non-burnt trees after the first fire. We then establish \eqref{eq-convergence_somme_arbres_brules_cas_sous_critique} which, by \eqref{eq-identification_xi_e_gaussienne_au_carre}, proves \autoref{thm-sous_crit}. Finally, we prove \autoref{pro-composante_geante_cas_sous_critique}.

\subsection{Configuration at the instant of the first fire}\label{section41}

Recall that we denote by $\zeta_n$ the first instant where an edge is set on fire during the dynamics on $\tn$. Then $\zeta_n$ is a truncated geometric random variable:
\begin{align*}
& \P(\zeta_n = \infty) = (1-p_n)^{n-1} \to 0 \text{ as } n \to \infty,
\\
\text{and } & \P(\zeta_n = k) = p_n (1-p_n)^{k-1} \quad\text{for every } k \in \{1, \dots, n-1\},
\end{align*}
and $p_n \zeta_n$ converges in distribution to an exponential random variable with parameter $1$. We work first on the event $\{\zeta_n = k\}$, $k \ge 1$ fixed. At the $k-1$-st step of the dynamics, we have deleted $k-1$ edges of $\tn$ uniformly at random to form a forest of $k$ trees $\t_{n,1}, \dots, \t_{n,k}$ where the labeling is made uniformly at random. We know from Lemma 5 in Bertoin \cite{Bertoin-Fires_on_trees} (see also Pavlov \cite{Pavlov-The_asymptotic_distribution_of_maximum_tree_size_in_a_random_forest} or Pitman \cite{Pitman-Coalescent_random_forests}) that the sizes of these $k$ subtrees are i.i.d. Borel(1) random variables conditioned to have sum $n$: for any $n_1, \dots, n_k \ge 1$ such that $n_1 + \dots + n_k = n$,
\begin{equation}\label{eq-loi_foret_uniforme}
\P(\vert\t_{n,1}\vert = n_1, \dots, \vert\t_{n,k}\vert = n_k) = \frac{(n-k)!}{k n^{n-k-1}} \prod_{j = 1}^k \frac{n_j^{n_j-1}}{n_j!}.
\end{equation}
Moreover, conditionally on the partition of $\{1, \dots, n\}$ induced by the $k$ subsets of vertices of these subtrees, the $\t_{n,i}$'s are independent uniform Cayley trees on their respective set of vertices. Recall that the Borel(1) distribution belongs to the domain of attraction of the stable law of index $1/2$ so that, taking a number of order $p_n^{-1}$ of i.i.d. such random variables, the sum is typically of order $p_n^{-2}$. Then, loosely speaking, conditioning this sum to be abnormally large, here of order $n$, essentially amounts to conditioning one single variable to be large, the others being almost unaffected. As we pick one subtree proportionally to its number of edges, the giant tree is set on fire with high probability and we are left with a collection of trees with sizes roughly given by i.i.d. Borel(1) random variables. These two features are formalized in the next proposition.

We denote by $\t_{n,1}, \dots, \t_{n,\zeta_n}$ the forest defined as above: we define $\t_{n,1}, \dots, \t_{n,k}$ conditionally on the event $\{\zeta_n = k\}$ and then average with respect to $\zeta_n$. Denote also by $\vert\t_{n,1}\vert^* \ge \dots \ge \vert\t_{n,\zeta_n}\vert^*$ a non-increasing rearrangement of the sizes of these trees. Let $\e$ be an exponential random variable with parameter $1$ and independently, $\sigma$ a subordinator distributed as \eqref{eq-definition_subordinateur_S}.

\begin{Pro}\label{pro-convergence_arbres_non_brules_sauts_subordinateur_cas_sous_critique}
We have
\begin{equation*}
\lim_{n \to \infty} p_n^2 (n - \vert\t_{n,1}\vert^*, \vert\t_{n,2}\vert^*, \dots, \vert\t_{n,\zeta_n}\vert^*) = \e^2 (\sigma(1), J_1, J_2, \dots)
\end{equation*}
in distribution.
\end{Pro}

\begin{proof}
The proof is similar to that of \autoref{pro-convergence_tailles_des_composantes_GWT_vers_masses_du_CRT}. Aldous and Pitman \cite{Aldous_Pitman-Standard_additive_coalescent}, equation (34), provide the convergence in distribution
\begin{equation*}
\lim_{n \to \infty} k_n^{-2} (n - \vert\t_{n,1}\vert^*, \vert\t_{n,2}\vert^*, \dots, \vert\t_{n,k_n}\vert^*) = (\sigma(1), J_1, J_2, \dots)
\end{equation*}
for any sequence $k_n = o(n^{1/2})$. Let $f \colon \ell^1(\R) \to \R$ be a continuous and bounded function and set for any $x > 0$
\begin{equation*}
F_n(x) \coloneqq \E\big[f\big(p_n^2 \big(n - \vert\t_{n,1}\vert^*, \vert\t_{n,2}\vert^*, \dots, \vert\t_{n,\lfloor x p_n^{-1}\rfloor}\vert^*\big)\big)\big],
\quad\text{and}\quad
F(x) \coloneqq \E\big[f\big(x^2 \big(\sigma(1), J_1, J_2, \dots\big)\big)\big].
\end{equation*}
The previous convergence yields $\lim_{n \to \infty} F_n(x_n) = F(x)$ whenever $\lim_{n \to \infty} x_n = x$. Using Skorohod's representation Theorem, we may suppose $\lim_{n \to \infty} p_n \zeta_n = \e$ almost surely, then $\lim_{n \to \infty} F_n(p_n \zeta_n) = F(\e)$ almost surely and the claim follows from Lebesgue's Theorem.
\end{proof}

From this result, we see that with high probability, the first burnt subtree has a size of order $n$ and the forest that we obtain by discarding this tree and the edges previously fireproof has a total size of order $p_n^{-2} = o(n)$. This already strengthens the result \ref{thm-transition_de_phase_sous} of the introduction. The fire dynamics then continue independently on each tree of this forest and the total number of fireproof vertices is the sum of the number of fireproof vertices in each component.

\subsection{Total number of fireproof vertices}\label{section42}

We now study the dynamics on the remaining forest after the first fire. We know from \autoref{pro-convergence_arbres_non_brules_sauts_subordinateur_cas_sous_critique} that with high probability, the largest trees have size of order $p_n^{-2}$ so that they are now critical for the dynamics which continue on each with $p_n = (p_n^{-2})^{-1/2}$. To see this, we slightly generalize the convergence \ref{thm-transition_de_phase_crit} of the introduction. Let $(\tn')_{n \ge 1}$ be a sequence of Cayley trees with size $\vert\tn'\vert \sim a p_n^{-2}$ as $n \to \infty$ for some $a > 0$ and define $I_n' = \Card\{i \in \tn': i \text{ is fireproof}\}$.

\begin{Lem}\label{lem-generalistion_thm_convergence_densite_sites_ignifuges_cas_sous_critique}
The law of $p_n^2 I_n'$ converges weakly to the distribution $\mu_a$ defined by \eqref{eq-generalisation_distribution_proportion_sites_ignifuges_dans_un_arbre_de_Cayley}.
\end{Lem}

\begin{proof}
The proof of Theorem~1 in Bertoin \cite{Bertoin-Fires_on_trees} shows that $\vert\tn'\vert^{-1} I_n'$, the proportion of fireproof vertices in $\tn'$, converges in distribution to $D(a^{1/2})$, as defined in \eqref{eq-distribution_D_infini}. Since $p_n^2 \vert\tn'\vert \to a$, we get $p_n^2 I_n' \to a D(a^{1/2})$ in distribution. One easily checks that the latter is distributed according to~$\mu_a$.
\end{proof}

Using \autoref{pro-convergence_arbres_non_brules_sauts_subordinateur_cas_sous_critique} and \autoref{lem-generalistion_thm_convergence_densite_sites_ignifuges_cas_sous_critique}, we can now prove \eqref{eq-convergence_somme_arbres_brules_cas_sous_critique} and so, \autoref{thm-sous_crit}.

\begin{proof}[Proof of \autoref{thm-sous_crit}]
Conditionally given $\zeta_n$, we write $(\t_{n,1}', \dots, \t_{n,\zeta_n}')$ for the trees obtained by deleting the first $\zeta_n-1$ fireproof edges, listed so that $\t_{n,1}'$ is the tree burnt at time $\zeta_n$ and $\vert\t_{n,2}'\vert \ge \dots \ge \vert\t_{n,\zeta_n}'\vert$. Note that
\begin{equation*}
I_n = \sum_{k=2}^{\zeta_n} \Card\{i \in \t_{n,k}': i \text{ is fireproof}\}.
\end{equation*}
From \autoref{pro-convergence_arbres_non_brules_sauts_subordinateur_cas_sous_critique}, we have
\begin{equation*}
\lim_{n \to \infty} p_n^2 (n-\vert\t_{n,1}'\vert, \vert\t_{n,2}'\vert, \dots, \vert\t_{n,\zeta_n}'\vert) = \e^2 (\sigma(1), J_1, J_2, \dots)
\end{equation*}
in distribution. Therefore, for any $\eps > 0$ there exists $N \in \N$ and then $n_0 \in \N$ such that
\begin{equation*}
\P\bigg(\sum_{k = N}^\infty \e^2 J_k >\eps\bigg) < \eps,
\quad\text{and for any } n \ge n_0,\quad
\P\bigg(\sum_{k = N+1}^{\zeta_n} p_n^2 \vert\t_{n,k}'\vert > \eps\bigg) < \eps.
\end{equation*}
Recall that, conditionally given $(\e^2 J_k)_{k \ge 1}$, $(X_k(\e))_{k \ge 1}$ is a sequence of independent random variables sampled according to $\mu_{\e^2 J_k}$ respectively, where for every $x > 0$, $\mu_x$ is the probability measure on $(0,x)$ given by \eqref{eq-generalisation_distribution_proportion_sites_ignifuges_dans_un_arbre_de_Cayley}. In particular, $X_k(\e) \le \e^2 J_k$ for every $k \ge 1$; we also have $\Card\{i \in \t_{n,k}': i \text{ is fireproof}\} \le \vert\t_{n,k}'\vert$. Then
\begin{equation*}
\P\bigg(\sum_{k = N}^\infty X_k(\e) >\eps\bigg) < \eps,
\ \text{and for any } n \ge n_0,\
\P\bigg(\sum_{k = N+1}^{\zeta_n} p_n^2 \Card\{i \in \t_{n,k}': i \text{ is fireproof}\} > \eps\bigg) < \eps.
\end{equation*}
Conditionally on the partition of $\{1, \dots, n\}$ induced by the subsets of vertices of the subtrees, the $\t_{n,k}'$'s are independent uniform Cayley trees on their respective set of vertices. \autoref{pro-convergence_arbres_non_brules_sauts_subordinateur_cas_sous_critique} and \autoref{lem-generalistion_thm_convergence_densite_sites_ignifuges_cas_sous_critique} thus yield
\begin{equation*}
\lim_{n \to \infty} \sum_{k = 2}^N p_n^2 \Card\{i \in \t_{n,k}': i \text{ is fireproof}\} = \sum_{k = 1}^{N-1} X_k(\e)
\quad\text{in distribution.}
\end{equation*}
Since the rests are arbitrary small with high probability, we get
\begin{equation*}
\lim_{n \to \infty} \sum_{k = 2}^{\zeta_n} p_n^2 \Card\{i \in \t_{n,k}': i \text{ is fireproof}\} = \sum_{k = 1}^\infty X_k(\e)
\quad\text{in distribution.}
\end{equation*}
The above convergence is \eqref{eq-convergence_somme_arbres_brules_cas_sous_critique}, \autoref{thm-sous_crit} then follows from \eqref{eq-identification_xi_e_gaussienne_au_carre}.
\end{proof}

Combined with the results of Bertoin \cite{Bertoin-Fires_on_trees}, \autoref{pro-convergence_arbres_non_brules_sauts_subordinateur_cas_sous_critique} and \autoref{lem-generalistion_thm_convergence_densite_sites_ignifuges_cas_sous_critique} also entail \autoref{pro-composante_geante_cas_sous_critique} about the size of the largest fireproof connected component.

\begin{proof}[Proof of \autoref{pro-composante_geante_cas_sous_critique}]
Fix $\eps > 0$ and $\delta \in (0, 1/2)$. Let $\sigma$ be a subordinator distributed as \eqref{eq-definition_subordinateur_S} and $\chi \in (0, \eps)$ such that the probability that $\sigma$ admits no jump larger than $\chi$ during the time interval $[0, 1]$ is less than $\delta$. Consider the subtrees of $\tn$ larger than $\chi p_n^{-2}$ when an edge is set on fire for the first time. From \autoref{pro-convergence_arbres_non_brules_sauts_subordinateur_cas_sous_critique}, we know that the number of such trees converges to the number of jumps larger than $\chi$ made by $\sigma$ before time 1. The latter is almost surely finite and non-zero with a probability greater than $1-\delta$. Now from \autoref{lem-generalistion_thm_convergence_densite_sites_ignifuges_cas_sous_critique}, these subtrees are critical and thus for each, from Corollary 1 and Proposition 1 of Bertoin \cite{Bertoin-Fires_on_trees}, the probability that there exists a fireproof component larger than $\eps p_n^{-2}$ tends to $0$ and the probability that there exists at least one larger than $n^{- \eps} p_n^{-2}$ tends to $1$. Therefore for any $n$ large enough, on the one hand there exists in $\tn$ a fireproof subtree larger than $n^{- \eps} p_n^{-2}$ and on the other hand there exists none larger than $\eps p_n^{-2}$, both with a probability at least $1-2\delta$. The claim follows since $\delta$ is arbitrary.
\end{proof}


\section{Asymptotic proportion of burnt vertices in the supercritical regime}\label{section5}

We finally consider the supercritical regime $n^{-1} \ll p_n \ll n^{-1/2}$ and prove \autoref{thm-sur_crit_1}. Recall that $\b_{n,1}, \dots, \b_{n,\kappa_n}$ denote the sizes of the burnt subtrees, listed in order of appearance. Let $(\e_i)_{i \ge 1}$ be a sequence of independent exponential random variables with parameter $1$ and for each $i \ge 1$, denote by $\gamma_i \coloneqq \e_1 + \dots + \e_i$. Let also $(\Gauss_i)_{i \ge 1}$ be a sequence of i.i.d. standard Gaussian random variables, independent of $(\gamma_i)_{i \ge 1}$. We shall prove the following result.

\begin{Thm}\label{thm-sur_crit_2}
We have
\begin{equation*}
\lim_{n \to \infty} (n p_n)^{-2} (\b_{n,1}, \dots, \b_{n,\kappa_n}) = (\gamma_i^{-2} \Gauss_i^2)_{i \ge 1}
\end{equation*}
in distribution for the $\ell^1$ topology.
\end{Thm}

\autoref{thm-sur_crit_1} follows as a corollary.

\begin{proof}[Proof of \autoref{thm-sur_crit_1}]
As a consequence of \autoref{thm-sur_crit_2}, we have the convergence of the sums:
\begin{equation*}
\lim_{n \to \infty} (n p_n)^{-2} B_n = \sum_{i = 1}^\infty \gamma_i^{-2} \Gauss_i^2
\qquad\text{in distribution.}
\end{equation*}
Note that the sequence $(\gamma_i)_{i \ge 1}$ is distributed as the atoms of a Poisson random measure on $(0, \infty)$ with intensity $\d x$, it follows readily that the sequence $(\gamma_i^{-2} \Gauss_i^2)_{i \ge 1}$ is distributed as the atoms of a Poisson random measure on $(0, \infty)$ with intensity $(2\pi x^3)^{-1/2} \d x$. The above limit is thus distributed as $\sigma(1)$ where $\sigma$ is the subordinator defined by \eqref{eq-definition_subordinateur_S}; \autoref{thm-sur_crit_1} finally follows from the well-known identity $\sigma(1) = \Gauss^{-2}$ in distribution.
%
\end{proof}

As discussed in the introduction, in order to prove \autoref{thm-sur_crit_2}, we first show the joint convergence of the first $j$ coordinates for any $j \ge 1$, and then that, taking $j$ large enough, the other coordinates are arbitrary small with high probability. We conclude in the same manner as in the proof of \autoref{thm-sous_crit}.

\subsection{Asymptotic size of the first burnt subtrees}\label{section51}

We first prove the convergence of the size of the first burnt subtree $\b_{n,1}$. As in the preceding section, we let the dynamics evolve until an edge is set on fire for the first time, denoting this random time by $\zeta_n$. The size of the tree that burns at this instant is distributed as one among $\zeta_n$ i.i.d. Borel(1) random variables conditioned to have sum $n$, chosen proportionally to its value minus~1. As we have seen, $p_n \zeta_n$ converges in distribution to an exponential random variable with parameter $1$, thus $\zeta_n$ is typically of order $p_n^{-1}$ and the sum of $\zeta_n$ i.i.d. Borel(1) random variables is of order $p_n^{-2}$. In the previous section, we considered $p_n^{-2} = o(n)$ and we have seen that conditioning these random variables to have sum $n$ essentially amounts to conditioning one to be of order $n$ (\autoref{pro-convergence_arbres_non_brules_sauts_subordinateur_cas_sous_critique}). The behavior is notoriously different when $n = o(p_n^{-2})$. As an example, Pavlov \cite{Pavlov-The_asymptotic_distribution_of_maximum_tree_size_in_a_random_forest}, Theorem 3, gives an asymptotic of the size of the largest subtree when one removes $k_n-1$ edges uniformly at random, with $n=o(k_n^2)$.

\begin{Lem}\label{lem-taille_premier_arbre_brule_cas_sur_critique}
As $n \to \infty$, $(n p_n)^{-2} \b_{n,1}$ converges in distribution to $\e^{-2} \Gauss^2$ where $\Gauss$ and $\e$ are independent, respectively standard Gaussian and exponential with parameter $1$ distributed.
\end{Lem}

\begin{proof}
We work conditionally on $\{\zeta_n = k_n\}$ with $k_n \sim c p_n^{-1}$, $c > 0$ arbitrary and prove the convergence in distribution $(n p_n)^{-2} \b_{n,1} \to c^{-2} \Gauss^2$. The general claim then follows as in the proof of \autoref{pro-convergence_arbres_non_brules_sauts_subordinateur_cas_sous_critique}. For any $\lambda \ge 0$, we write
\begin{align*}
\E[\exp(-\lambda (n p_n)^{-2} \b_{n,1})] &= \sum_{m = 0}^\infty \exp(-\lambda (n p_n)^{-2} m) \P(\b_{n,1} = m)
\\
&= \int_0^\infty \exp(-\lambda (n p_n)^{-2} \lfloor x \rfloor) \P(\b_{n,1} = \lfloor x \rfloor) \d x
\\
&= \int_0^\infty \exp(-\lambda (n p_n)^{-2} \lfloor x (n p_n)^2 \rfloor) \P(\b_{n,1} = \lfloor x (n p_n)^2 \rfloor) (n p_n)^2 \d x.
\end{align*}
We show the pointwise convergence of the densities
\begin{equation*}
\lim_{n \to \infty} \P(\b_{n,1} = \lfloor x (n p_n)^2 \rfloor) (n p_n)^2 = \frac{c}{\sqrt{2 \pi x}} \exp\Big(-\frac{c^2 x}{2}\Big),
\quad\text{for every } x > 0.
\end{equation*}
Then Scheffé's Lemma implies that this convergence also holds in $L^1$, which allows us to pass to the limit in the above integral:
\begin{equation*}
\lim_{n \to \infty} \E[\exp(-\lambda (n p_n)^{-2} \b_{n,1})]
= \int_0^\infty \exp(-\lambda x) \frac{c}{\sqrt{2 \pi x}} \exp\Big(-\frac{c^2 x}{2}\Big) \d x
= \E[\exp(-\lambda c^{-2} \Gauss^2)].
\end{equation*}
Recall the distribution of $k$ i.i.d. Borel(1) random variables conditioned to have sum $n$: for any integers $n_1, \dots, n_k \ge 1$ such that $n_1 + \dots + n_k = n$,
\begin{equation*}
\P(\beta_{n,1} = n_1, \dots, \beta_{n,k} = n_k) = \frac{(n-k)!}{k n^{n-k-1}} \prod_{j = 1}^k \frac{n_j^{n_j-1}}{n_j!}.
\end{equation*}
In particular, the $\beta_{n,j}$'s are identically distributed and for any $m \in \{1, \dots, n-k+1\}$, summing over all the $n_2, \dots, n_k \ge 1$ such that $n_2 + \dots + n_k = n-m$,
\begin{equation*}
\P(\beta_{n,1} = m)
= \sum_{n_2, \dots, n_k} \frac{(n-k)!}{k n^{n-k-1}} \frac{m^{m-1}}{m!} \prod_{j = 2}^k \frac{n_j^{n_j-1}}{n_j!}
= \frac{(n-k)!}{k n^{n-k-1}} \frac{m^{m-1}}{m!} \frac{(k-1) (n-m)^{n-m-k}}{(n-m-k+1)!}.
\end{equation*}
We know that conditionally given $(\beta_{n,1}, \dots, \beta_{n,k})$, $\b_{n,1}$ is distributed as one of these variables picked proportionally to its value minus $1$. Thus for any $m \in \{2, \dots, n-k+1\}$, we have
\begin{align*}
\P(\b_{n,1} = m) &= \sum_{j=1}^k \P(\b_{n,1} = \beta_{n,j} \,\vert\, \beta_{n,j} = m)\, \P(\beta_{n,j} = m)
\\
&= k \frac{m-1}{n-k} \P(\beta_{n,1} = m)
\\
&= (m-1)(k-1) \frac{(n-k-1)!}{n^{n-k-1}} \frac{m^{m-1}}{m!} \frac{(n-m)^{n-m-k}}{(n-m-k+1)!}.
\end{align*}
Suppose that $m, k \to \infty$ as $n \to \infty$ with $m, k = o(n)$, then Stirling's formula yields
\begin{equation*}
\P(\b_{n,1} = m)
= \frac{1}{\sqrt{2\pi}} \, \frac{k}{n\sqrt{m}} \exp\bigg(- \frac{k^2m}{2n^2} + O\bigg(\frac{k^3m}{n^3}\bigg) + O\bigg(\frac{(km)^2}{n^3}\bigg)\bigg) (1+o(1)).
\end{equation*}
For any $x, c > 0$, taking $m = \lfloor x (n p_n)^2\rfloor$ and $k \sim c p_n^{-1}$, we get
\begin{align*}
\P(\b_n = \lfloor x (n p_n)^2\rfloor) (n p_n)^2
&= \frac{c}{\sqrt{2 \pi x}} \Big(-\frac{c^2 x}{2} + O\Big(\frac{1}{n p_n}\Big) + O(n p_n^2)\Big) (1+o(1))
\\
&= \frac{c}{\sqrt{2 \pi x}} \exp\Big(-\frac{c^2 x}{2}\Big) (1+o(1)),
\end{align*}
and the proof is now complete.
\end{proof}

More generally, for any integer $j \ge 1$, denote by $\zeta_{n,j}$ the time of the $j$-th fire, so that $\b_{n,j}$ denotes the size of the subtree burnt at time $\zeta_{n,j}$.

\begin{Pro}\label{pro-taille_premiers_arbres_brules_cas_sur_critique}
We have for any $j \ge 1$,
\begin{equation*}
\lim_{n \to \infty} (n p_n)^{-2} (\b_{n,1}, \dots, \b_{n,j}) = \big(\gamma_1^{-2} \Gauss_1^2, \dots, \gamma_j^{-2} \Gauss_j^2\big)
\quad\text{in distribution.}
\end{equation*}
\end{Pro}

\begin{proof}
We prove the claim for $j=2$ for simplicity of notation, the general case follows by induction in the same manner. Notice first that the times at which the first $j$ fires appear jointly converge:
\begin{equation}\label{eq-convergence_temps_des_premiers_feux}
\lim_{n \to \infty} p_n (\zeta_{n,1}, \dots, \zeta_{n,j}) = (\gamma_1, \dots, \gamma_j)
\quad\text{in distribution.}
\end{equation}
Indeed, conditionally given the size of the first burnt subtree $\b_{n,1} = m$ and the number of edges previously fireproof $\zeta_{n,1}-1 = k-1$, it remains a forest containing $(n-1) - (k-1) - (m-1) = n-m-k+1$ edges and the time $\zeta_{n,2}-\zeta_{n,1}$ we wait for the second fire after the first one is again a truncated geometric random variable which takes value
\begin{align*}
\infty &\text{ with probability } (1-p_n)^{n-m-k+1},
\\
\text{and } \ell &\text{ with probability } p_n (1-p_n)^{\ell-1} \text{, for any } \ell \in \{1, \dots, n-m-k+1\}.
\end{align*}
Since $\b_{n,1}+\zeta_{n,1} = o(n)$ in probability, we see that $p_n (\zeta_{n,2}-\zeta_{n,1})$, conditionally given $\b_{n,1}$ and $\zeta_{n,1}$, converges in distribution to an exponential random variable with parameter $1$. This yields \eqref{eq-convergence_temps_des_premiers_feux} in the case $j = 2$.

The same idea gives the claim of \autoref{pro-taille_premiers_arbres_brules_cas_sur_critique}. The remaining forest after the first fire is, conditionally given $\b_{n,1}$ and $\zeta_{n,1}$, uniformly distributed amongst the forests with $\zeta_{n,1}-1$ trees and $n-\b_{n,1}$ vertices. Therefore, conditionally given $\b_{n,1}$ and $\zeta_{n,1}$, $\b_{n,2}$ is distributed as the size of a tree chosen at random with probability proportional to its number of edges in a forest consisting of $\zeta_{n,2}-1$ trees with total size $n-\b_{n,1} \sim n$. Then the proof of \autoref{lem-taille_premier_arbre_brule_cas_sur_critique} shows that such a random variable, rescaled by a factor $(n p_n)^{-2}$, converges in distribution to $\gamma_2^{-2} \Gauss_2^2$. This yields 
\begin{equation*}
\lim_{n \to \infty} (n p_n)^{-2} \big(\b_{n,1}, \b_{n,2}\big) = \big(\gamma_1^{-2} \Gauss_1^2, \gamma_2^{-2} \Gauss_2^2\big)
\quad\text{in distribution,}
\end{equation*}
and the proof is complete after an induction on $j$.
\end{proof}

\subsection{Asymptotic size of all burnt subtrees}\label{section52}

To strengthen the convergence from finite dimensional vectors to the $\ell^1$ convergence, we need to bound the remainders. This is done in the following lemma, the last ingredient for the proof of \autoref{thm-sur_crit_1}.
\begin{Lem}\label{lem-borne_restes_serie_taille_arbres_brules}
For any $\eps > 0$, we have
\begin{equation*}
\lim_{j_0 \to \infty} \limsup_{n \to \infty} \P\bigg((n p_n)^{-2} \sum_{j = j_0}^\infty \b_{n,j} > \eps\bigg) = 0.
\end{equation*}
\end{Lem}

In order to prove this result, we consider a slightly different sequence of random subtrees of $\tn$, which can be coupled with the sequence of burnt subtrees and for which the study is easier. Precisely, consider the following random dynamics on $\tn$: we remove successively the edges in a random uniform order and at each step, we mark one subtree at random proportionally to its number of edges. We stress that in this procedure, the subtrees are not burnt, which implies that the edges of a marked subtree can be removed afterward and that a subtree of a marked one may be marked as well. For each $k = 1, \dots, n-2$, we denote by $\b_{n,k}'$ the size of the subtree which is marked when $k$ edges have been removed.

\begin{Lem}\label{lem-borne_restes_esperance_arbres_marques}
There exists a numerical constant $C > 0$ such that for any $a \in (0, \infty)$, we have
\begin{equation*}
\limsup_{n \to \infty} n^{-2} p_n^{-1} \sum_{k = \lfloor a p_n^{-1} \rfloor}^{\lfloor n - u_n \rfloor} \E[\b_{n,k}'] \le \frac{C}{a},
\end{equation*}
where $u_n = n \exp(-\sqrt{n p_n})$ for every integer $n$.
\end{Lem}

The role of the sequences $u_n$ and $a p_n^{-1}$ shall appear in the proofs of \autoref{lem-borne_restes_esperance_arbres_marques} and \autoref{lem-borne_restes_serie_taille_arbres_brules}; note that since $\lim_{n \to \infty} n p_n = \infty$, we have
\begin{equation}\label{eq-convergence_suite_technique}
\lim_{n \to \infty} p_n u_n = 0
\quad\text{and}\quad
\lim_{n \to \infty} (n p_n)^{-1} \ln\bigg(\frac{n-u_n}{u_n}\bigg)
= 0.
\end{equation}
The first convergence shows that the sum in \autoref{lem-borne_restes_esperance_arbres_marques} is not empty for $n$ large enough.

\begin{proof}
Fix $k \le n-2$ and let $(\beta_{n,1}, \dots, \beta_{n,k})$ be a $k$-tuple formed by i.i.d. Borel(1) random variables conditioned to have sum $n$. As we have seen, $\b_{n,k}'$ can be viewed as one the $\beta_{n,i}$'s picked at random with probability proportional to its value minus one and hence,
\begin{equation*}
\E[\b_{n,k}']
= \E\bigg[\sum_{i=1}^k \beta_{n,i} \frac{\beta_{n,i} - 1}{n-k}\bigg]
= \frac{n}{n-k} \E\bigg[\sum_{i=1}^k (\beta_{n,i}-1) \frac{\beta_{n,i}}{n}\bigg].
\end{equation*}
Bertoin \cite{Bertoin-Burning_cars_in_a_parking_lot}, Section 3.1, provides an upper bound for the expectation on the right-hand side. Precisely, Proposition~1 in \cite{Bertoin-Burning_cars_in_a_parking_lot}, together with Lemma~5 and equation~(2) there, shows that there exists a numerical constant $K > 0$ such that for every integers $1 \le k \le n$, we have
\begin{equation*}
\E\bigg[\sum_{i=1}^k (\beta_{n,i}-1) \frac{\beta_{n,i}}{n}\bigg]
\le K \bigg(\frac{n}{k}\bigg)^2.
\end{equation*}
Hence for every $n$,
\begin{equation*}
\sum_{k = \lfloor a p_n^{-1} \rfloor}^{\lfloor n - u_n \rfloor} \E[\b_{n,k}']
\le K n^3 \sum_{k = \lfloor a p_n^{-1} \rfloor}^{\lfloor n - u_n \rfloor} \frac{1}{k^2 (n-k)}.
\end{equation*}
Comparing sums and integrals, we have on the one hand,
\begin{equation*}
\sum_{k = \lfloor a p_n^{-1} \rfloor}^{\lfloor 3n/4 \rfloor} \frac{1}{k^2 (n-k)}
\le \frac{4}{n} \sum_{k = \lfloor a p_n^{-1} \rfloor}^{\lfloor 3n/4 \rfloor} \frac{1}{k^2}
= \frac{4}{a} n^{-1} p_n (1+o(1)),
\end{equation*}
and on the other hand,
\begin{equation*}
\sum_{k = \lceil 3n/4 \rceil}^{\lfloor n - u_n \rfloor} \frac{1}{k^2 (n-k)}
\le n^{-2}\bigg[\ln(x) - \ln(n-x) - \frac{n}{x}\bigg]_{3n/4}^{n - u_n}
= n^{-2} \ln\bigg(\frac{n - u_n}{u_n}\bigg) (1+o(1)).
\end{equation*}
Summing the two terms and appealing \eqref{eq-convergence_suite_technique}, we obtain
\begin{equation*}
\sum_{k = \lfloor a p_n^{-1} \rfloor}^{\lfloor n - u_n \rfloor} \E[\b_{n,k}']
\le \frac{4K}{a} n^2 p_n (1+o(1)),
\end{equation*}
and the claim follows.
\end{proof}

We have a natural coupling between burnt and marked subtrees, which enables us to deduce \autoref{lem-borne_restes_serie_taille_arbres_brules} from \autoref{lem-borne_restes_esperance_arbres_marques}: for each $k = 1, \dots, n-2$, we toss a coin which gives Head with probability $p_n$; the first burnt subtree, say, $\t_{n, 1}$, is distributed as the first marked subtree, say, $\t_{n,k}'$, for which the outcome is Head. Then, the second burnt subtree $\t_{n, 2}$ is distributed as the next marked subtree for which the outcome is Head and which is not contained in $\t_{n,k}'$, and so on. In the next proof, we implicitly assume that the marked and burnt subtrees are indeed coupled.

\begin{proof}[Proof of \autoref{lem-borne_restes_serie_taille_arbres_brules}]
Fix $\eps, \delta > 0$ and $a > \delta^{-1}$. Since $p_n \zeta_{n,j} \to \gamma_j$ in distribution as $n \to \infty$ for any $j \ge 1$ and $j^{-1} \gamma_j \to 1$ in probability as $j \to \infty$ from the law of large numbers, we may, and do, fix $j_0 \ge 1$ and further $n_0 \ge 1$ such that for any $n \ge n_0$, we have
\begin{equation*}
\P(\zeta_{n,j_0} > a p_n^{-1}) \ge 1-\delta.
\end{equation*}
For any $j \ge 1$, denote by $\theta_{n,j}-1$ the number of edges that have been removed in the marking procedure when we mark the subtree corresponding to the burnt subtree $\b_{n,j}$. We have
\begin{equation*}
\sum_{j \ge j_0} \b_{n,j}
\le \sum_{k = 1}^{n - 2} \b_{n,k}' \mathbf{1}_{\{\eta_k = 1\}} \mathbf{1}_{\{k \ge \theta_{n,j_0}\}},
\end{equation*}
where $\eta_k = 1$ if and only if the outcome of the coin which is tossed at the $k$-th step is Head. Further, since $\zeta_{n,1} = \theta_{n,1}$ and $\zeta_{n,j} \le \theta_{n,j}$ for every $j \ge 2$, we see that
\begin{equation*}
\P\bigg((n p_n)^{-2} \sum_{j = j_0}^\infty \b_{n,j} > \eps \,\bigg\vert\, \zeta_{n,j_0} > a p_n^{-1}\bigg)
\le \P\bigg((n p_n)^{-2} \sum_{k = \lfloor a p_n^{-1} \rfloor}^{n - 2} \b_{n,k}' \mathbf{1}_{\{\eta_k = 1\}} > \eps\bigg).
\end{equation*}
Recall that $\lim_{n \to \infty} p_n u_n = 0$, which implies that the probability that no tree is marked after the $\lfloor n - u_n \rfloor$-th step is $(1-p_n)^{\lceil u_n \rceil - 2} \ge 1-\delta$ for any $n$ large enough. Finally, from \autoref{lem-borne_restes_esperance_arbres_marques},
\begin{equation*}
\limsup_{n \to \infty} \P\bigg((n p_n)^{-2} \sum_{k = \lfloor a p_n^{-1} \rfloor}^{\lfloor n - u_n \rfloor} \b_{n,k}' \mathbf{1}_{\{\eta_k = 1\}} > \eps\bigg)
\le \limsup_{n \to \infty} \eps^{-1} (n p_n)^{-2} \sum_{k = \lfloor a p_n^{-1} \rfloor}^{\lfloor n - u_n \rfloor} \E[\b_{n,k}'] p_n
\le \eps^{-1} \frac{C}{a}.
\end{equation*}
We conclude that 
\begin{equation*}
\limsup_{n \to \infty} \P\bigg((n p_n)^{-2} \sum_{j = j_0}^\infty \b_{n,j} > \eps\bigg)
\le \eps^{-1} C\delta + 2 \delta,
\end{equation*}
and the claim follows since $\delta$ is arbitrary.
\end{proof}

Using the same reasoning as in the proof of \autoref{thm-sous_crit}, \autoref{thm-sur_crit_2} follows readily from \autoref{pro-taille_premiers_arbres_brules_cas_sur_critique} and \autoref{lem-borne_restes_serie_taille_arbres_brules}.

\begin{proof}[Proof of \autoref{thm-sur_crit_2}]
For every $n, j \ge 1$, we write $((n p_n)^{-2} \b_{n,k})_{k \ge 1} = S_n(j) + R_n(j)$ with
\begin{equation*}
S_n(j) = (n p_n)^{-2} \big(\b_{n,1}, \b_{n,2}, \dots, \b_{n,j}, 0, 0, \dots\big),
\quad\text{and}\quad
R_n(j) = (n p_n)^{-2} \big(0, \dots, 0, \b_{n,j+1}, \b_{n,j+2}, \dots\big);
\end{equation*}
and similarly, $(\gamma_i^{-2} \Gauss_i^2)_{i \ge 1} = S(j) + R(j)$, with
\begin{equation*}
S(j) = \big(\gamma_1^{-2} \Gauss_1^2, \gamma_2^{-2} \Gauss_2^2, \dots, \gamma_j^{-2} \Gauss_j^2, 0, 0, \dots\big),
\quad\text{and}\quad
R(j) = \big(0, \dots, 0, \gamma_{j+1}^{-2} \Gauss_{j+1}^2, \gamma_{j+2}^{-2} \Gauss_{j+2}^2, \dots\big).
\end{equation*}
From \autoref{pro-taille_premiers_arbres_brules_cas_sur_critique}, for any $j \ge 1$, $\lim_{n \to \infty} S_n(j) = S(j)$ in distribution. Further, for any $\eps > 0$, since the sequence $(\gamma_i^{-2} \Gauss_i^2)_{i \ge 1}$ is summable, and thanks to \autoref{lem-borne_restes_serie_taille_arbres_brules}, there exists $j_0 \ge 1$ and then $n_0 \ge 1$ such that
\begin{equation*}
\P(\|R(j_0)\| > \eps) < \eps, \quad\text{and for every } n \ge n_0, \quad \P(\|R_n(j_0)\| > \eps) < \eps.
\end{equation*}
which completes the proof.
\end{proof}


\vspace{\baselineskip}
\textbf{Acknowledgement. }
The author wishes to thank Jean Bertoin for introducing the problem and fruitful discussions, and also two anonymous referees whose suggestions and remarks helped to improve this paper.\newline
This work is supported by the Swiss National Science Foundation 200021\_144325/1.



\begin{thebibliography}{10}

\bibitem{Aldous-The_continuum_random_tree_3}
D.~Aldous.
\newblock The continuum random tree {III}.
\newblock {\em Ann. Probab.}, 21(1):248--289, 1993.

\bibitem{Aldous_Pitman-Brownian_bridge_asymptotics_for_random_mappings}
D.~Aldous and J.~Pitman.
\newblock Brownian bridge asymptotics for random mappings.
\newblock {\em Random Structures \& Algorithms}, 5(4):487--512, 1994.

\bibitem{Aldous_Pitman-Standard_additive_coalescent}
D.~Aldous and J.~Pitman.
\newblock The standard additive coalescent.
\newblock {\em Ann. Probab.}, 26(4):1703--1726, 1998.

\bibitem{Bertoin-Burning_cars_in_a_parking_lot}
J.~Bertoin.
\newblock Burning cars in a parking lot.
\newblock {\em Communications in Mathematical Physics}, 306(1):261--290, 2011.

\bibitem{Bertoin-Fires_on_trees}
J.~Bertoin.
\newblock {Fires on trees}.
\newblock {\em Ann. Inst. H. Poincar\'e Probab. Statist.}, 48(4):909--921,
  2012.

\bibitem{Caratheodory-Theory_of_functions_of_a_complex_variable}
C.~Carath{{\'e}}odory.
\newblock {\em Theory of Functions of a Complex Variable}.
\newblock AMS Chelsea Publishing Series. American Mathematical Society, 2001.

\bibitem{Chaumont_Uribe_Bravo-Markovian_bridges_weak_continuity_and_pathwise_constructions}
L.~Chaumont and G.~Uribe~Bravo.
\newblock Markovian bridges: weak continuity and pathwise constructions.
\newblock {\em Ann. Probab.}, 39(2):609--647, 2011.

\bibitem{Pavlov-The_asymptotic_distribution_of_maximum_tree_size_in_a_random_forest}
Y.~L. Pavlov.
\newblock The asymptotic distribution of maximum tree size in a random forest.
\newblock {\em Theor. Probab. Appl.}, 22(3):509--520, 1977.

\bibitem{Perman-Order_statistics_for_jumps_of_normalised_subordinators}
M.~Perman.
\newblock Order statistics for jumps of normalised subordinators.
\newblock {\em Stochastic Processes and their Applications}, 46(2):267--281,
  1993.

\bibitem{Pitman-Coalescent_random_forests}
J.~Pitman.
\newblock Coalescent random forests.
\newblock {\em Journal of Combinatorial Theory, Series A}, 85(2):165--193,
  1999.

\bibitem{Pitman-Poisson_Kingman_partitions-test}
J.~Pitman.
\newblock Poisson-{K}ingman partitions.
\newblock In {\em Statistics and science: a {F}estschrift for {T}erry {S}peed},
  volume~40 of {\em IMS Lecture Notes Monogr. Ser.}, pages 1--34. Institute of
  Mathematical Statistics, 2003.

\bibitem{Pitman_Yor-The_two_parameter_Poisson_Dirichlet_distribution_derived_from_a_stable_subordinator}
J.~Pitman and M.~Yor.
\newblock The two-parameter {Poisson-Dirichlet} distribution derived from a
  stable subordinator.
\newblock {\em Ann. Probab.}, 25(2):855--900, 1997.

\end{thebibliography}
\end{document}